\documentclass[oneside,reqno]{amsart}
\usepackage{amsfonts,amsmath,latexsym,verbatim,amscd,mathrsfs,color,array}

\usepackage{amsmath,amssymb,amsthm,amsfonts,graphicx,color}
\usepackage{amssymb}
\usepackage{pdfsync}
\usepackage{epstopdf}

\newcommand{\ass}{\quad\mbox{as}\quad}

\long\def\hide#1{}

\newcommand{\inn}{{\quad\hbox{in } }}

\newcommand{\ttt}{\tilde }

\newcommand{\LL}{{\mathcal L}  }
\newcommand{\TT}{{\mathcal T}  }

\newcommand{\nn}{ {\nabla}  }

\newcommand{\pp}{ {\partial} }

\newcommand{\vp}{\varphi}

\newcommand{\RR}{{{\mathcal R}}}

\newcommand{\R} {\mathbb R}
\newcommand{\Z} {\mathbb Z}
\newcommand{\cuad}{{\sqcap\kern-.68em\sqcup}}

\newcommand{\foral}{\quad\mbox{for all}\quad}
\newcommand{\ve}{\varepsilon}

\newcommand{\be}{\begin{equation}}
\newcommand{\ee}{\end{equation}}

\newcommand{\equ}[1]{(\ref{#1})}

\renewcommand{\div}{{\rm div}}
\newcommand{\curl}{\mathop{\rm curl}}

\newtheorem{lemma}{Lemma}[section]
\newtheorem{prop}{Proposition}[section]
\newtheorem{theorem}{Theorem}

\newtheorem{remark}{Remark}[section]
\newcommand{\bremark}{\begin{remark} \em}
\newcommand{\eremark}{\end{remark} }

\numberwithin{equation}{section}

\begin{document}

\title[Travelling helices and the vortex filament conjecture]{Travelling helices and the vortex filament conjecture in the incompressible  Euler equations }

\author[J. D\'avila]{Juan D\'avila}
\address{\noindent  J.D.: Department of Mathematical Sciences University of Bath,
Bath BA2 7AY, United Kingdom
}
\address{\noindent Instituto de Matem\'aticas, Universidad de Antioquia, Calle 67, No. 53--108, Medell\'\i{}n, Colombia
}
\email{jddb22@bath.ac.uk}

\author[M. del Pino]{Manuel del Pino}
\address{\noindent M.dP.:  Department of Mathematical Sciences University of Bath,
Bath BA2 7AY, United Kingdom}
\email{m.delpino@bath.ac.uk}

\author[M. Musso]{Monica Musso}
\address{\noindent M.M.:  Department of Mathematical Sciences University of Bath,
Bath BA2 7AY, United Kingdom}
\email{m.musso@bath.ac.uk}

\author[J. Wei]{Juncheng Wei}
\address{\noindent J.W.: Department of Mathematics University of British Columbia, Vancouver, BC V6T 1Z2, Canada
}  \email{jcwei@math.ubc.ca}

\maketitle

\begin{abstract}
We consider the Euler equations in $\R^3$ expressed in vorticity form
$$\left\{
\begin{aligned}
\vec \omega_t  +
({\bf u}\cdot \nn ){\vec \omega}
&=( \vec \omega \cdot \nn ) {\bf u}  &&   \\
\quad {\bf u}  = \curl \vec \psi,\ &
%{\bf u}  =&  \curl \vec \psi
-\Delta \vec \psi =  \vec \omega . &&
\end{aligned} \right.
$$
A classical question that goes back to Helmholtz is to describe the evolution of solutions with a high concentration around a curve. The work of Da Rios in 1906 states that such a curve must evolve by
the so-called binormal curvature flow.
%$
%\gamma_\tau = \gamma_s\times \gamma_{ss}$, $\tau = %|\log\ve |t $, where $\ve $ is the radius of the region where the solution is concentrated.
Existence of true solutions concentrated near a given curve that evolves by this law
is a long-standing open question that has only
been answered for the special case of a circle travelling with constant speed along its axis, the thin vortex-rings.
We provide what appears to be the first rigorous construction of {\em helical filaments}, associated to a translating-rotating helix. The solution is defined at all times and does not change form with time. The result generalizes to multiple polygonal helical filaments travelling and rotating together.

\end{abstract}

\section{Introduction}

We consider the 3-dimensional Euler equations for an ideal incompressible homogeneous fluid  in a time  interval $(0,T)$ and a smooth initial condition $u_0$  given by
\be \label{euler0}
\begin{aligned}
{\bf u}_t  + ({\bf u}\cdot \nn ){\bf u} &= -\nn p  &&\inn \R^3 \times (0,T) , \\
{\rm div}\, {\bf u} &= 0  && \inn \R^3\times (0,T) , \\
{\bf u}(\cdot,0)  &= u_0 && \inn \R^3 .
\end{aligned}
\ee
For a solution  ${\bf u }$ of  \equ{euler0},
its vorticity is defined as
$ \vec\omega = \curl {\bf u}  $. Then $\vec\omega$ solves the Euler system in {\em vorticity form}  \equ{euler0},
\be \label{euler}
\begin{aligned}
\vec \omega_t  +
({\bf u}\cdot \nn ){\vec \omega}
&=( \vec \omega \cdot \nn ) {\bf u}  &&  \inn \R^3\times (0,T), \\
\quad {\bf u}  = \curl \vec \psi,\ &
%{\bf u}  =&  \curl \vec \psi    \inn \R^3\times (0,T)
-\Delta \vec \psi =  \vec \omega  && \inn \R^3\times (0,T),\\
{\vec \omega}(\cdot,0)  &=  \curl u_0 && \inn \R^3 .
\end{aligned}
\ee
We are interested in solutions of the Euler equations whose vorticities are large and uniformly concentrated near an evolving smooth curve $\Gamma(t)$ embedded in entire $\R^3$ and so that the associated velocity field vanishes as the distance to the curve goes to infinity. This type of solutions, {\em vortex filaments}, are classical objects of study in fluid dynamics, whose analysis traces back to Helmholtz and Kelvin. In 1867 Helmholtz considered with great attention the situation where the vorticity is concentrated in a circular vortex filament with small section. He detected that these {\em vortex rings}  have an approximately steady form and travel with a large constant velocity along the axis of the ring. In 1894, Hill found an explicit axially symmetric solution of \equ{euler} supported in a sphere (Hill's vortex ring).

In 1970, Fraenkel \cite{fraenkel} provided the first construction of a vortex ring concentrated around  a torus with fixed radius with a small, nearly singular section $\ve>0$, travelling with constant speed $\sim |\log\ve |$, rigorously establishing the behaviour predicted by Helmholtz. Vortex rings have been analyzed in larger generality in  \cite{fraenkel-berger,norbury,ambrosetti-struwe,devaleriola-vanschaftingen}.

\medskip
 Da Rios  \cite{darios}  in 1906, and  Levi-Civita \cite{levicivita1908} in 1908, formally found the general
law of motion of a vortex filament with a thin section of radius $\ve>0$, uniformly distributed around an evolving curve $\Gamma(t)$. Roughly speaking, he demonstrated that
under suitable assumptions on the  solution, the curve
evolves by its {\em binormal flow}, with a large velocity of order $|\log\ve|$. More precisely, if $\Gamma(t)$ is parametrized as  $x= \gamma(s,t)$ where $s$ designates its arclength parameter, then $\gamma(s,t)$ asymptotically obeys a law of the form
\be\label{bin}
\gamma_t =  2c|\log\ve|  (\gamma_s\times \gamma_{ss})
\ee
as $\ve \to 0$, or scaling $t= |\log\ve|^{-1}\tau $,
\be\label{bin1}
\gamma_\tau =  2c (\gamma_s\times \gamma_{ss})= 2c\kappa {\bf b}_{\Gamma(\tau)}  .
\ee
Here $c$ corresponds to the {\em circulation} of the velocity field on the boundary of sections to the filament, which is assumed to be a constant independent of $\ve$. For the curve  $\Gamma(\tau)$ parametrized as $x=\gamma(\tau,s)$,  we designate by ${\bf t}_{\Gamma(\tau)},\, {\bf n}_{\Gamma(\tau)},\, {\bf b}_{\Gamma(\tau)}$ the usual tangent, normal and binormal unit vectors, $\kappa$ its curvature.

\medskip
The law \equ{bin} was formally rediscovered several times during the 20th century, see the survey paper by Ricca \cite{ricca}.

\medskip
In  \cite{jerrard-seis}, Jerrad and Seis gave a precise form to Da Rios' computation under the weakest known conditions on a solution to \equ{euler} which remains suitably concentrated around an evolving curve. Their result considers a solution  $\vec \omega_\ve(x,t)$  under a set of conditions that imply that  as $\ve \to 0$
\be\label{conv}
\vec \omega _\ve (\cdot,|\log\ve|^{-1}\tau)\rightharpoonup  \delta_{\Gamma(\tau)}  {\bf t}_{\Gamma(\tau)},  \quad 0\le \tau \le T ,
\ee
where $\Gamma(\tau)$ is a sufficiently regular curve
and $\delta_{\Gamma(\tau)}$ denotes a uniform Dirac measure on the curve.
They prove that $\Gamma(\tau)$ does indeed evolve by the law \equ{bin1}.  See \cite{jerrard-smets} and its references for results on the flow \equ{bin1}.

\medskip
The existence of a true solution of \equ{euler} satisfying \equ{conv}  near a given curve $\Gamma(\tau)$
that evolves by the binormal flow \equ{bin1} is an outstanding open question called the {\em vortex filament conjecture}. This statement is unknown except for very special cases.  A basic example  is given by a circle $\Gamma(\tau)$ with radius $R$
translating with constant speed along its axis parametrized as
\begin{align}
\label{circle}
\gamma(s,\tau)  =   \left ( \begin {matrix}
  R \cos \big (\frac  {s}R\big ) \\
R \sin\big(\frac { s}R \big ) \\ \frac {2c}R   \tau \end{matrix} \right ), \quad \ 0<s \le  2\pi R.
\end{align}
We check that
$$ \gamma_s \times \gamma_{ss} = R^{-1}  \left ( \begin {matrix}
  0 \\
0 \\ 1  \end{matrix} \right ) ,
$$
and hence $\gamma$ satisfies equation \equ{bin1}.
Fraenkel's result \cite{fraenkel}
on  thin
vortex rings yields existence of a solution $\vec\omega_\ve(x,t)$ without change of form, such that the asymptotic behavior \equ{conv} holds for this curve. In other words, the statement of the {\em vortex filament conjecture} holds true
for the case of the traveling circle.
The vortex filament conjecture is also true in the case of a straight line. It suffices to consider a concentrated steady state constrained to a plane normal to the line and then trivially
extending it in the remaining variable.

\medskip
Another known solution of the binormal flow \equ{bin1} that does not change its form in time is the {\em rotating-translating helix}, the curve $\Gamma(\tau)$ parametrized as
\begin{equation}
\label{helix2}
\gamma (s,\tau) =
\left( \begin{matrix}
R \cos \big(
\frac{s-a_1 \tau}{\sqrt{h^2 + R^2}} \big)
\\
R \sin \big(
\frac{s-a_1 \tau}{\sqrt{h^2 + R^2}} \big)
\\
\frac{h s + b_1 \tau }{\sqrt{h^2+ R^2}}
\end{matrix} \right),
\quad a_1 =   \frac{2 c h}{ h^2 + R^2 },\quad  b_1 = \frac{2 c  R^2}{ h^2 + R^2 }.
\end{equation}
\hide{
\begin{equation}
\label{helix2}
\gamma (\tau,s) =
\left( \begin{matrix}
R e^{ i
\frac{s-b \tau}{\sqrt{h^2 + R^2}}}
\\
\frac{h s + a \tau }{\sqrt{h^2+ R^2}}
\end{matrix} \right),
\quad a = \frac{2 c  R^2}{ h^2 + R^2 } , \quad
b =   \frac{2 c h}{ h^2 + R^2 },
\end{equation}}
%where we identify $ \R^3 = \C \times \R$ and write $(x',z)\in  \R^3$, $x'\in \C$, $z\in \R$.
Here $R>0$, $h \ne  0$ are constants that are in correspondence  with the curvature and torsion of the helix, the numbers respectively given by
$ \frac{R}{h^2+R^2}$ and $ \frac{h}{h^2+R^2} $. We readily check that the parametrization \equ{helix2} satisfies equation \equ{bin1}.
This helix degenerates into the
 the traveling circle \equ{circle} when $h\to 0$ and to a straight line when $|h|\to +\infty$.

\medskip

The helices \eqref{helix2}
were observed to possibly describe vortex filaments in the sense
\equ{conv} more than 100 years ago by Joukowsky \cite{joukowsky}, Da Rios \cite{darios1916} and Levi-Civita \cite{levicivita1932}. Their no-change in form can be described by using the rotation matrices
\begin{align}
\label{Q}
P_\theta  = \left [   \begin{matrix}   \cos \theta & -\sin \theta \\  \sin \theta & \phantom{-}\cos \theta  \end{matrix}    \right ], \quad
Q_\theta   =  \left [   \begin{matrix}  P_\theta & 0 \\  0 & 1  \end{matrix}    \right ].
\end{align}
\hide{and the  affine transformation constituted by a rotation in the two first coordinates and a translation in the third one
\begin{align*}
S_{\theta, \sigma } ( x )   \, := \,   \left [    \begin{matrix}  P_\theta x' \\  x_3 + \sigma \end{matrix}    \right ]  =   Q_\theta x   + \left [   \begin{matrix}   0\\  \sigma  \end{matrix}    \right ], \quad x'=  \left [   \begin{matrix}   x_1\\ x_2 \end{matrix}    \right ].  % \quad  x' = \left [   \begin{matrix}   x_1 \\  x_2\end{matrix}    \right ].
\end{align*}}
Then $\gamma(s,\tau)$ has can be recovered from  $\gamma(s,0)$ by a rotation and a vertical translation by means of the sense that
$$\gamma (s,\tau)=
Q_{- a_2 \tau} \gamma(s,0)  + \left[ \begin{matrix}
0 \\  b_2 \tau
\end{matrix} \right],
%S_{- a_2 \tau, b_2 \tau } \gamma(s,0) ,
$$
\be\label{12}
a_2 =   \frac{2 c h}{ (h^2 + R^2)^{\frac 32}},\quad
b_2 = \frac{2 c  R^2}{ (h^2 + R^2)^{\frac 32} }.
\ee
Renewed interest in {\em helical filaments} has risen in the last two decades, see \cite{velasco} for a recent survey. However, there has been no proof of their existence.

\medskip
The purpose of this paper is {\bf to construct} a true helical filament $\vec \omega_\ve(x,t)$ satisfying \equ{conv}.
This solution  does not change of form and goes along the helix in the sense that the vector field \be\label{F}  F(x,\tau) = \vec \omega_\ve (x,  |\log\ve|^{-1}\tau ),  \quad \ee satisfies
\be \label{conds}\begin{aligned}
F(x,\tau)\  =&  \  Q_{a_2 \tau} F (   P_{- a_2 \tau}x', x_3 + b_2 \tau , 0), \\
F(x, \tau )\ =&\ F(x', x_3 + 2\pi h , \tau )
\end{aligned}
\ee
with $a_2,b_2$ given by \equ{12},  $x=(x',x_3),\ x'=(x_1,x_2)$.

\begin{theorem}\label{teo1} Let $\Gamma(\tau)$ be the helix parametrized by equation $\equ{bin1}$. Then
there exists a smooth solution $\vec \omega_\ve (x,t)$ to $\equ{euler}$, defined for $t\in (-\infty, \infty)$
that does not change form and follows the helix in the sense $\equ{F}$-$\equ{conds}$, such that for all $\tau$,
$$
\vec\omega_\ve (x, \tau|\log\ve|^{-1}) \rightharpoonup  c\delta_{\Gamma(\tau)}{\bf t}_{\Gamma(\tau)} \ass \ve \to 0.
$$

\end{theorem}
This result extends to the situation of several helices symmetrically arranged.  Let us consider the curve $\Gamma(\tau)$
parametrized by $\gamma(\tau ,s)$ in \equ{helix2}. Let us define for $j=1,\ldots, k $  the curves $\Gamma_j(\tau)$ parametrized by
\be \gamma_j(s,\tau) =  Q_{ 2\pi \frac{j-1} k } \gamma(s,\tau ) . \label{helixj}
\ee
The following result extends that of Theorem \ref{teo1}.

\begin{theorem}\label{teo2} Let $\Gamma_j(\tau)$ be the helices parametrized by equation $\equ{helixj}$, for $j=1,\ldots,k$. Then
there exists a smooth solution $\vec \omega_\ve (x,t)$ to $\equ{euler}$, defined for $t\in (-\infty, \infty)$
that does not change form and follows the helix in the sense $\equ{F}$-$\equ{conds}$, such that for all $\tau$,
$$
\vec\omega_\ve (x, \tau|\log\ve|^{-1}) \rightharpoonup  c \sum_{j=1}^k \delta_{\Gamma_j(\tau)}{\bf t}_{\Gamma_j(\tau)} \ass \ve \to 0.
$$

\end{theorem}

Our construction takes advantage of a screw-driving symmetry invariance enjoyed by the Euler equations observed by Dutrifoy \cite{dutrifoy} and  Ettinger-Titi \cite{ettinger-titi}, see also the more recent results in \cite{bnl,jln}.
From the analysis in the latter reference, it follows that a function of the form
\begin{align*}
\vec\omega(x,t)  =  \frac{1}{h} w(  P_{-\frac {x_3}h} x', t)
\left (\begin{matrix}   P_{\frac \pi 2} x' \\ h    \end{matrix}  \right) , \qquad  x'=  (x_1,x_2) ,
\end{align*}
%where $w(x',t)$ is a scalar function, ,
solves \equ{euler} if the scalar function $w(x',t)$ satisfies the transport equation
\begin{equation}\label{PB-00}
\begin{aligned}
\left\{
\begin{aligned}
w_t + \nabla^\perp \psi \cdot \nabla w &=0 && {\mbox {in}} \quad \R^2 \times (0, T)
\\
- {\mbox {div}} (K \nabla \psi) &= w && {\mbox {in}} \quad \R^2 \times (0, T),
\end{aligned}
\right.
\end{aligned}
\end{equation}
where $(a,b)^\perp = (b,-a)$ and
$K (x_1,x_2) $ is the matrix
\begin{equation}
\label{defK-0}
K(x_1 , x_2 ) = \frac{1}{h^2+x_1^2+x_2^2}
\left(
\begin{matrix}
h^2+x_2^2 & -x_1  x_2\\
-x_1 x_2 & h^2+x_1^2
\end{matrix}
\right).
\end{equation}
For completeness we prove this fact in    Section~\ref{sectHelical}.

\medskip

The proof of Theorems~\ref{teo1} and \ref{teo2} is reduced to finding solutions  of \eqref{PB-00} concentrated at the vertices of a rotating $k$-regular polygon, which do not change form.
We devote the rest of the paper to build such a solution by means of elliptic singular perturbation techniques.

\medskip

Solutions concentrated near helices in other PDE settings have
have been built in \cite{contreras-jerrard,chiron,ddmr,jerrard-smets-gp,wei-yang}.
Construction of vortex filaments with small vorticities around general sets has been achieved in \cite{enciso-peralta}. Connected with Theorem \ref{teo2},
the formal law for the dynamics of nearly parallel interacting filaments in the Euler equation has been found in \cite{klein-majda-damodaran}, which is the same law governing almost parallel vortex filaments for the
Gross-Pitaevskii
equation, see \cite{jerrard-smets-gp}.
The filaments in Theorem~\ref{teo2} are not nearly parallel.
The motion of vortex filaments is the natural generalization of the motion of point vortices for the 2D incompressible Euler equations. Their desingularization has been rigorously analyzed in \cite{marchioro-pulvirenti,serfati1,serfati2,Z} and \cite{ddmw}.  Nonlinear stability of point vortices has been recently established in
\cite{ionescu-jia}.

\section{Solutions with helical symmetry}
\label{sectHelical}

In this section we recall how to find solutions of the incompressible Euler equations in 3d with helical symmetry, following Dutrifoy \cite{dutrifoy} and  Ettinger-Titi \cite{ettinger-titi}.

Let $h>0$.
%We will use the notation $x\in \R^3$, $ x = (x_1,x_2,x_3) = (x',x_3)$ with $x'=(x_1,x_2)\in \R^2$ and $x_3 \in \R$.
We say that a scalar function $f:\R^3\to \R$ has helical symmetry if
\begin{align}
\label{sym1}
f( P_\theta x' , x_3+ h\theta) = f(x',x_3)  \foral \theta\in \R , \ \  x = (x',x_3) \in \R^3 ,
\end{align}
%{\color{red} This is incompatible with the helices \eqref{helix2}.}
where $P_\theta$ is defined in \eqref{Q}.
For a scalar function satisfying \eqref{sym1} we have $f(x) = f( P_{- \frac{x_3}{h}}x',0)$, and so $f$ is determined by its values on the horizontal plane $$\{ \, x = (x',x_3) \ |  \ x_3 = 0 \, \}. $$
Also, if $f$ is $C^1$, then $f$ satisfies \eqref{sym1} if and only if $\nabla f \cdot \vec \xi = 0$ where $\vec \xi $ is the vector field
\[
\vec\xi(x)= (- x_2,x_1,h) , \qquad x =(x_1,x_2,x_3).
\]
A vector field $F:\R^3\to \R^3$ is said to have helical symmetry if
\begin{align}
\label{sym2}
F( P_{\theta} x' , x_3+ h\theta) = Q_\theta F(x',x_3), \quad \forall \theta\in \R , \quad \forall x = (x',x_3) \in \R^3 ,
\end{align}
where $Q_\theta$ is the matrix defined in \eqref{Q}.
If $F$ satisfies \eqref{sym2}	then
\[
F ( x) = R_{\frac{x_3}{h}} F( P_{-  \frac{x_3}{h}}x',0) ,
\]
so that $F$ is determined by its values on the plane $\{ \, x = (x',x_3) \ |  \ x_3 = 0 \, \}$.
Again, if $F = (F_1,F_2,F_3)$ is $C^1$, then it satisfies \eqref{sym2} if and only if
\begin{align}
\label{helicalV}
\nabla F_1 \cdot \vec \xi = - F_2, \quad \nabla F_2 \cdot \vec \xi = F_1, \quad
\nabla F_3 \cdot \vec \xi = 0 .
\end{align}

The following result is a consequence from the analysis in \cite{ettinger-titi}.
\begin{lemma}
Suppose that  $w(x',t)$ satisfies the transport equation
\eqref{PB-00}.
Then there exists a vector field ${\bf u}$ with helical symmetry \eqref{sym2} such that
\begin{align}
\label{omegaHelical}
\vec \omega(x,t)  =  \frac{1}{h} w( P_{-\frac {x_3}h} x',t)  \vec\xi(x) , \quad x = (x',x_3)
\end{align}
satisfies the Euler equation \eqref{euler}.
\end{lemma}
\begin{proof}
We note that the vorticity $\vec\omega$ defined by \eqref{omegaHelical} has the helical symmetry \eqref{sym2}.
Moreover the scalar function $w$ is satisfies \eqref{sym1} and hence $\nabla w \cdot \vec \xi=0$.

The velocity vector field ${\bf u}$ associated to $\omega$ will be constructed such that it has the helical symmetry  \eqref{sym2} and satisfies in addition
\begin{align}
\label{uPerpXi}
{\bf u}\cdot \vec \xi  = 0 .
\end{align}
This condition and  \eqref{omegaHelical} say that $\vec\omega$ and ${\bf u}$ are always orthogonal.

%Next we construct a velocity field ${\bf u} = ( u_1,u_2,u_3)$ so that \eqref{euler} holds and such that it satisfies the desired symmetries.
To construct ${\bf u} = ( u_1,u_2,u_3)$, we first define $u_1(x',0)$, and $u_2(x',0)$ by the following relation
\begin{align}
\label{uFromPsi}
\left(
\begin{matrix}
u_1\\u_2
\end{matrix}
\right)
=  \frac{1}{h^2+x_1^2+x_2^2}
\left(
\begin{matrix}
- x_1  x_2 &   h^2 +  x_1^2 \\
- h^2 - x_2^2 & x_1 x_2
\end{matrix}
\right)
\left(
\begin{matrix}
\partial_{x_1} \psi \\ \partial_{x_2} \psi
\end{matrix}
\right) .
\end{align}
Next we define $u_3(x',0)$ using \eqref{uPerpXi}:
\[
u_3 = \frac{1}{h}( x_2 u_1  - x_1 u_2 ) .
\]
In this way ${\bf u}(x',0)$ is defined for all $x' \in \R^2$ and then we extend ${\bf u}$ to $\R^3$ imposing that  \eqref{sym2} is satisfied.

Let us explain formula \eqref{uFromPsi}. The incompressibility condition $\nabla \cdot {\bf u} = 0$ is rewritten as
\begin{align*}
0 & = \partial_{x_1} u_1 + \partial_{x_2} u_2 + \partial_{x_3} u_3
= \partial_{x_1} u_1 + \partial_{x_2} u_2  + \frac{y}{h} \partial_{x_1} u_3-  \frac{x}{h} \partial_{x_2} u_3 ,
\end{align*}
where we have used that $- x_2 u_1 + x_1 u_2 + h u_3 = 0$ (third formula in \eqref{helicalV}).
Then from \eqref{uPerpXi} we get
\begin{align*}
0 &=
\partial_{x_1} u_1 + \partial_{x_2} u_2  + \frac{y}{h^2} \partial_{x_1} ( x_2 u_1 - x_1 u_2) -  \frac{x}{h^2 } \partial_{x_2}  ( x_2 u_1 - x_1 u_2)
\\
&=
\frac{1}{h^2} \partial_{x_1}[ ( h^2 + x_2^2) u_1 - x_1 x_2 u_2 ]
+ \frac{1}{h^2} \partial_{x_2} [ ( -x_1 x_2 u_1 + (h^2 + x_1^2) u_2 ] .
\end{align*}
This motivates to take $\psi:\R^2 \to \R$ such that
\begin{align}
\label{psiFromU}
\begin{aligned}
\psi_{x_1} &= \phantom{-}\frac{1}{h^2} [ - x_1 x_2 u_1 + (h^2 + x_1^2) u_2 ]
\\
\psi_{x_2} &= - \frac{1}{h^2} [  (h^2 + x_2^2)  u_1 -  x_1 x_2 u_2 ]  .
\end{aligned}
\end{align}
This is equivalent to \eqref{uFromPsi}. Later on we verify that $\psi$ satisfies $-\div ( K \nabla \psi) = w$.

We claim that $\omega $, ${\bf u}$ satisfy the Euler equations \eqref{euler}. We check the first equation component by component and on the plane $x_3=0$. The equality for all $x_3$ follows from the helical symmetry.
First we note that
\begin{align}
\label{vortStretch}
( \vec \omega \cdot \nabla  ) {\bf u}  = \frac{1}{h} w
\left( \begin{matrix}
-u_2 \\ \phantom{-}u_1 \\  \phantom{-}0
\end{matrix}\right)
\end{align}
by the form of $\omega$ in \eqref{omegaHelical} and the equations \eqref{helicalV} satisfied by ${\bf u}$.
Using \eqref{psiFromU} and \eqref{PB-00}  we get
\[
w_t - \frac{1}{h^2}[ - (h^2+x_2^2) u_1 + x_1 x_2 u_2] w_{x_1}
+ \frac{1}{h^2}[ -x_1x_2u_1 + (h^2+x_1^2) u_2 ] w_{x_2} = 0 ,
\]
which gives, using that $\nabla w \cdot \vec \xi = 0$ and ${\bf u} \cdot \xi = 0$, that
\[
w_t + u_1 w_{x_1} + u_2 w_{x_2} + u_3 w_{x_3} = 0.
\]
By \eqref{vortStretch} this is  the third component in the first equation in \eqref{euler}.
The first and second components are handled similarly.

Next we verify that $\vec \omega = \curl {\bf u}$.
Indeed, we look first at the third component of $\curl {\bf u}$, which is
\begin{align*}
\partial_{x_1} u_2 - \partial_{x_2} u_1
&= \partial_{x_1} \Big[ \frac{1}{h^2+x_1^2+x_2^2}
( - ( h^2  + x_2^2)  \psi_{x_1} +  x_1 x_2 \psi_{x_2})
\Big]
\\
& \quad
- \partial_{x_2} \Big[ \frac{1}{h^2+x_1^2+x_2^2}
( - x_1 x_2 \psi_{x_1} + ( h^2  + x_1^2)  \psi_{x_2})
\Big]
\\
& = - \div ( K \nabla \psi )
\\ &= w.
\end{align*}
The first and second component of  $\curl  {\bf u}$ are computed similarly.

\end{proof}

\section{Equation \texorpdfstring{\eqref{PB-00}}{PB-00} in a rotational frame}
\label{sectRot}

To find a solution  $\vec \omega_\ve (x,t)$ as described in Theorem~\ref{teo1},
of the form
\begin{align*}
\vec\omega_\varepsilon(x,t)  =  \frac{1}{h} w(  P_{-\frac {x_3}h} x', t)
\left (\begin{matrix}   P_{\frac \pi 2} x' \\ h    \end{matrix}  \right) , \qquad  x'=  (x_1,x_2) ,
\end{align*}
we need to find a solution $w(x',t) $ of \eqref{PB-00} such that, with the change of variable $t = |\log\varepsilon|^{-1}\tau$, it is concentrated near a single point in the plane that is rotating with constant speed around the origin.
Here $P$ is the matrix defined in \eqref{Q}.
In terms of $\tau$, equation \eqref{PB-00} becomes
\begin{equation}\label{PB}
\begin{aligned}
\left\{
\begin{aligned}
|\log \ve | w_\tau + \nabla^\perp \psi \cdot \nabla w &=0 &&  {\mbox {in }}  \R^2 \times (-\infty, \infty)
\\
- \div ( K \nabla \psi)  &= w &&  {\mbox {in }}  \R^2 \times (-\infty, \infty),
\end{aligned}
\right.
\end{aligned}
\end{equation}
where
%\begin{align}
%\label{defL}
%L\psi = \div ( K \nabla \psi),
%\end{align}
%and
$K $ was defined in \eqref{defK-0}.
%\begin{equation*}
%\label{defK-0}
%K(x_1 , x_2 ) = \frac{1}{h^2+x_1^2+x_2^2}
%\left(
%\begin{matrix}
%h^2+x_2^2 & -x_1  x_2\\
%-x_1 x_2 & h^2+x_1^2
%\end{matrix}
%\right).
%\end{equation*}

%In Lemma 2.17 \cite{et} it is proved that, if $w$, $\psi$ solve
%system \eqref{PB}, then there is a solution to 3D Euler with helical symmetry.
%The parameter $h$ can be scaled out. Indeed, if $(w_h, \psi_h)$ solve \eqref{PB}, then
%$$
%w (x_1 , x_2 ,\tau) = h^2 w_h (|h| x_1 , |h| x_2,\tau) , \quad \psi (x_1 , x_2) = \psi_h (|h| x_1 , |h| x_2 )
%$$
%solve \eqref{PB} with $h^2= 1$ in the definition of the matrix $K$.

For notational simplicity, in what follows we will restrict ourselves to the case
$$
h^2 =1.
$$

\medskip
Let $\alpha$ be a fixed constant. We look for rotating solutions to problem \eqref{PB} of the form
\begin{equation}
\label{ansaztRot}
w (x', \tau) = W \left( P_{ \alpha \tau } x' \right), \quad \psi (x',\tau) = \Psi \left( P_{ \alpha \tau } x' \right) ,
\end{equation}
where $P_\theta$ is defined in \eqref{Q}.

Let $\tilde x = P_{\alpha \tau} x'$.
For the rest of this computation we will denote differential operators with respect to $\tilde x$ with a subscript $\tilde x$, and differential operators with respect to $x'$ without any subscript.
In terms of $(W, \Psi)$, the second equation in \eqref{PB} becomes
$$
-\div_{\tilde x} ( K(\tilde x) \nabla_{\tilde x} \psi)  \Psi = W.
$$
Let us see how the first equation gets transformed.
%Define the $2\times2$ matrices
%$$
%%P= \left[ \begin{matrix} e^{i \alpha t} & i e^{i \alpha t} \end{matrix} \right] , \quad
%J= \left[ \begin{matrix} 0 & 1 \\ -1 & 0\end{matrix} \right]
%$$
Observe that
$$
w_\tau = \nabla_{\tilde x} W  \cdot (\alpha P_{\frac\pi2+\alpha \tau } x' )
%= \nabla_z W \cdot (i \alpha x')
= - \alpha \nabla_{\tilde x} W \cdot \tilde x^\perp,
$$
and
$$
(\nabla w)^T = (\nabla_{\tilde x} W)^T P_{\alpha \tau}, \quad
(\nabla \psi)^T = (\nabla_{\tilde x} \Psi )^T P_{\alpha \tau}.
$$
Let $J = P_{\frac{\pi}{2}}$.
Since $\nabla \psi^\perp \cdot \nabla w=
(\nabla w )^T J \nabla \psi$,
and $ P_{\alpha\tau} J P_{\alpha\tau}^T = J $,
we have
\begin{align*}
\nabla \psi^\perp \cdot \nabla w &= (\nabla w )^T J \nabla \psi = (\nabla_{\tilde x} W)^T P_{\alpha\tau} J P_{\alpha\tau}^T \nabla_{\tilde x} \Psi \\
&= (\nabla_{\tilde x} W )^T J \nabla_{\tilde x} \Psi = (\nabla_{\tilde x} \Psi)^\perp \cdot (\nabla_{\tilde x} W ) .
\end{align*}
We conclude that equation $|\log \ve |w_\tau + \nabla^\perp \psi \cdot \nabla w =0$ becomes
$$
-\alpha |\log \ve |  \nabla_{\tilde x} W(\tilde x) \cdot \tilde x^\perp +  \nabla_{\tilde x}^\perp \Psi \cdot \nabla_{\tilde x} W= 0,
$$
or equivalently
\begin{align}
\label{eq00}
\nabla_{\tilde x} W \cdot \nabla_{\tilde x}^\perp \left( \Psi - \alpha |\log \ve | \frac{|\tilde x|^2 }{ 2} \right) = 0 .
\end{align}
Now, if $W(\tilde x) = F(\Psi(\tilde x) - \alpha  |\log \ve |
\frac{ |\tilde x|^2}{2} )$, for some function $F$, then automatically \eqref{eq00} holds.
We conclude that if $\Psi$ is a solution to
\begin{equation}\label{PB10}
- \div_{\tilde x} \cdot (K(\tilde x) \nabla_{\tilde x} \Psi ) = F (\Psi - \frac{\alpha}{2} |\log \ve | |\tilde x|^2 ) \quad {\mbox {in}} \quad \R^2,
\end{equation}
for some function $F$, and $W$ is given by
$$
W(\tilde x)= F(\Psi(\tilde x) - \alpha  |\log \ve | \frac{|\tilde x|^2}{2} )
$$
then $ (w, \psi )$ defined by \eqref{ansaztRot} is a solution for \eqref{PB}.

In \eqref{PB10} we take
$$
F (s) = \ve^2 f(s) , \quad f(s) = e^s,
$$
where $\varepsilon>0$.
In the sequel we write $\tilde x $ as $x = (x_1,x_2) \in \R^2$ and consider the equation
\[
-\div (K \nabla \Psi ) = \ve^2 e^{(\Psi -{\alpha \over 2} |\log \ve | |x|^2 )} \quad {\mbox {in}} \quad \R^2,
\]
where $\varepsilon>0$.

\section{Construction of an approximate solution}
\label{approx}

The rest of the paper is devoted to find a  solution to
\begin{equation}\label{PB1}
-\nabla \cdot (K \nabla \Psi ) = \ve^2 e^{(\Psi -{\alpha \over 2} |\log \ve | |x|^2 )} \quad {\mbox {in}} \quad \R^2,
\end{equation}
where $\varepsilon>0$ is small, and such that the solution is concentrated near a fixed point $q_0 = (x_0 , 0)$ with $x_0>0$. The parameter $x_0$ is fixed and corresponds to $R$ in the definition of the helices \eqref{helix2}. The number $\alpha$ corresponds to the angular velocity of the rotating solution described in Section~\ref{sectRot}, and will be adjusted suitably in the course of the proof.

We start with the construction of a global approximate solution to
\eqref{PB1}.
Towards this end,  consider the change of variables
\begin{equation*}
%\label{change}
x_1 = x_0 + z_1 , \quad x_2 = \sqrt{1+ x_0^2} z_2.
\end{equation*}

Let
\begin{equation}\label{Gammaep}
\Gamma_\ve (z) = \log {8 \over (\ve^2 + |z|^2 )^2}
\end{equation}
be the Liouville profile which satisfies
\begin{equation}\label{uno}
\Delta \Gamma_\ve  + {8 \ve^2 \over (\ve^2 + |z|^2)^2 }= 0 \quad {\mbox {in }} \quad \R^2.
\end{equation}

For a fixed $\delta >0$, we define
\begin{equation}\label{psi1}
\Psi_{1\ve} (x ) = {\alpha \over 2} |\log \ve | x_0^2 - \log (1+ x_0^2) + \Gamma_\ve (z)  \left( 1+ c_1 z_1  \right)  \,
\end{equation}
in the region  $|z| < \delta$.
Here $c_1$  is a constant to be determined later on.

For a function $\Psi$, we define the error-function as
\begin{equation}\label{defS}
S [\Psi ] (z) =   L \Psi + \ve^2 \, f (\Psi -{\alpha \over 2} |\log \ve | ( |x_0+ z_1|^2 + (1+ x_0^2) |z_2|^2) ) ,
\end{equation}
where
\begin{align}
\label{defL}
L \Psi = \div ( K \nabla \Psi).
\end{align}

We would like to describe  explicitly the error function $S[\Psi]$, for $\Psi = \Psi_{1\ve}$, in the region $|z| < \delta$.
Using the notation $K = \left( \begin{matrix} K_{11} & K_{12} \\ K_{12} & K_{22} \end{matrix} \right)$, we compute in the original variables
$$
\begin{aligned}
L= \nabla \cdot (K \nabla \cdot) &= K_{11} \, \pp_{x_1}^2 + K_{22} \, \pp_{x_2}^2 + 2 K_{12} \,  \pp_{x_1 x_2}^2\\
&+ ( \pp_{x_1} K_{11}   + \pp_{x_2} K_{12} ) \,  \pp_{x_1} + ( \pp_{x_2} K_{22} + \pp_{x_1} K_{12}) \, \pp_{x_2}
\end{aligned}
$$
Letting $r^2 = |x|^2$, we  get that
\begin{equation}\label{exL}
\begin{aligned}
L& = {1+ x_2^2 \over 1+ r^2} \pp_{x_1 x_1} + {1+ x_1^2 \over 1+ r^2} \pp_{x_2 x_2} - 2{ x_1 x_2 \over 1+r^2} \pp_{x_1 x_2}\\
 &+ \left( \pp_{x_1} ({1+x_2^2 \over 1+r^2} ) - \pp_{x_2} ({x_1 x_2 \over 1+r^2} ) \right) \pp_{x_1}\\
 &+ \left( \pp_{x_2} ({1+x_1^2 \over 1+r^2} ) - \pp_{x_1} ({x_1 x_2 \over 1+r^2} ) \right) \pp_{x_2} .
\end{aligned}
\end{equation}
Since
$$
\pp_{x_1} ({1+x_2^2 \over 1+r^2} ) - \pp_{x_2} ({x_1 x_2 \over 1+r^2} )= - {x_1 \over 1+ r^2} ({2\over 1+r^2} +1)
$$
and
$$
\pp_{x_2} ({1+x_1^2 \over 1+r^2} ) - \pp_{x_1} ({x_1 x_2 \over 1+r^2} )= - {x_2 \over 1+ r^2} ({2\over 1+r^2} +1)
$$
formula \eqref{exL} simplifies to
\begin{equation}\label{exL1}
\begin{aligned}
L& = {1+ x_2^2 \over 1+ r^2} \pp_{x_1 x_1} + {1+ x_1^2 \over 1+ r^2} \pp_{x_2 x_2} - 2{ x_1 x_2 \over 1+r^2} \pp_{x_1 x_2}\\
 &- {x_1 \over 1+ r^2} ({2\over 1+r^2} +1) \pp_{x_1} - {x_2 \over 1+ r^2} ({2\over 1+r^2} +1) \pp_{x_2} .
\end{aligned}
\end{equation}
In the $z$-variable, from \eqref{exL1} we read the operator $L$
\begin{equation}\label{exL2}
L = {1\over 1+ x_0^2} \left[   \pp_{z_1 z_1} + \pp_{z_2 z_2} + B_0 \right]
\end{equation}
where
$$
B_0 [\phi] = \pp_{z_i} \left( b_{ij}^0 \pp_{z_j} \phi \right), \quad {\mbox {and}} \quad
b_{ij}^0 (z) = \left( K_{ij} - {1\over 1+ x_0^2} \delta_{ij} \right).
$$
For later purpose, we observe that, in a region where $z$ is bounded, say $|z| < \delta $, for a fixed $\delta $ small, the operator $B_0$ has the form
\begin{equation}\label{B0f}
\begin{aligned}
B_0 &= - \left( {2 x_0 \over 1+ x_0^2} z_1+ O(|z|^2) \right) \pp_{z_1 z_1} + O(|z|^2) \pp_{z_2 z_2} - \left( 2 x_0 z_2 + O(|z|^2 ) \right) \pp_{z_1 z_2} \\
&- \left( x_0 (1+{2\over 1+ x_0^2} ) + O(|z|) \right) \pp_{z_1} + O(|z| ) \pp_{z_2}.
\end{aligned}
\end{equation}
Using \eqref{exL2} and the explicit form of the nonlinearity $f$, we obtain the  explicit expression for the error function $S[\Psi]$ defined in \eqref{defS} is
\begin{align*}%\label{expS}
(1+ x_0^2) S[\Psi ] (z) &= (\pp_{z_1 z_1} + \pp_{z_2 z_2} ) \Psi + B_0 [\Psi] \\
& \quad  + \ve^2   (1+ x_0^2 ) e^{-{\alpha \over 2} |\log \ve | x_0^2} e^\Psi  e^{-{\alpha \over 2} |\log \ve | \ell_0(z) } ,
\end{align*}
where $\ell_0(z) $ is defined by
\begin{equation}\label{elle0}
\ell_0(z) = 2 x_0 z_1 + z_1^2 + (1+ x_0^2) z_2^2 .
\end{equation}
Using \eqref{Gammaep} and \eqref{uno}, we have that
$$
(\pp_{z_1 z_1} + \pp_{z_2 z_2} ) (\Gamma_\ve ) +\ve^2 (1+ x_0^2 ) e^{-{\alpha \over 2} |\log \ve | x_0^2} e^{\Gamma_\ve + {\alpha \over 2} |\log \ve | x_0^2 - \log (1+ x_0^2)} = 0, \quad {\mbox {in}} \quad \R^2.
$$
Thus, if for a moment  we take
$$
c_1 = 0
$$
in the definition of $\Psi_{1\ve}$ in \eqref{psi1},
we get
$$
\begin{aligned}
 (1+ x_0^2) S[\Psi_{1\ve}  ] (z) &= B_0 [\Gamma_\ve]
 + \ve^2     e^{\Gamma_\ve }  \left[  e^{-{\alpha \over 2} |\log \ve | \ell_0(z)} -1 \right],
\end{aligned}
$$
where $\ell_0$ is defined in \eqref{elle0}.
Using \eqref{B0f}, we write $B_0 [\Gamma_\ve ]$ as follows
\be \label{E1}
B_0 [ \Gamma_\ve ] = - {2 x_0 \over 1+x_0^2} z_1 \pp_{z_1 z_1} \Gamma_\ve - 2 x_0 z_2 \pp_{z_1 z_2} \Gamma_\ve - x_0 (1+ {2\over 1+ x_0^2} ) \pp_{z_1} \Gamma_\ve + E_1,
\ee
where $E_1$ is a smooth  function, uniformly bounded for $\ve $ small, in a bounded region for $z$.
The constant $c_1$ in \eqref{psi1} will be chosen to partially cancel the part of the error given by $-{2 x_0 \over 1+x_0^2} z_1 \pp_{z_1 z_1} \Gamma_\ve - 2 x_0 z_2 \pp_{z_1 z_2} \Gamma_\ve - x_0 (1+ {2\over 1+ x_0^2} ) \pp_{z_1} \Gamma_\ve$.
Using the explicit expression of $\Gamma_\ve$,
\begin{align*}
\pp_1 \Gamma_\ve (z) = -{4z_1 \over \ve^2 +|z|^2} , &\quad z_1 \pp_{11} \Gamma_\ve (z) = -{4z_1 \over \ve^2 +|z|^2} +{8z_1^3 \over (\ve^2 +|z|^2)^2}\\
z_2 \pp_{12} \Gamma_\ve (z) &= {8z_2^2 z_1 \over (\ve^2 +|z|^2)^2}
\end{align*}
 and the identity
$$
z_1 z_2^2  = { |z|^2 z_1 \over 4} - { {\mbox {Re} } (z^3 )\over 4}
$$
we obtain
\begin{align*}
-&{2 x_0 \over 1+x_0^2} z_1 \pp_{z_1 z_1} \Gamma_\ve - 2 x_0 z_2 \pp_{z_1 z_2} \Gamma_\ve - x_0 (1+ {2\over 1+ x_0^2} ) \pp_{z_1} \Gamma_\ve \\
&=4x_0 (1+{4\over 1+x_0^2}) {z_1 \over \ve^2 + |z|^2} -4x_0 (1+ {3\over 1+x_0^2} ) {|z|^2 z_1 \over (\ve^2 + |z|^2)^2} + {4x_0^3 \over 1+x_0^2} {{\mbox {Re} } (z^3 ) \over (\ve^2 + |z|^2)^2}\\
&= {4x_0 \over 1+ x_0^2} {z_1 \over \ve^2 + |z|^2} + {4 x_0^3 \over 1+ x_0^2} {{\mbox {Re} } (z^3 ) \over (\ve^2 + |z|^2)^2} + {4 x_0 (4+ x_0^2) \over 1+ x_0^2}  { \ve^2 \, z_1 \over (\ve^2 + |z|^2 )^2}
\end{align*}
We shall choose the constant $c_1$ to eliminate the first term of the last line in the error, using   the fact that
$$
(\pp_{z_1 z_1} + \pp_{z_2 z_2} ) (c_1 z_1 \Gamma_\ve ) = - 8 \, c_1 \, {z_1 \over \ve^2 + |z|^2} - 8\,  c_1 \, {\ve^2  z_1 \over (\ve^2 + |z|^2 )^2}.
$$
We take $c_1$ in \eqref{psi1} to be
\begin{equation}\label{defc1}
c_1 ={1\over 2}  {x_0 \over 1+ x_0^2}.
\end{equation}
With this choice of $c_1$ in the definition of $\Psi_{1\ve}$ in \eqref{psi1}, we get
$$
\begin{aligned}
(1+ x_0^2) S[\Psi_{1\ve}  ] (z) &= - (8 c_1 -{4 x_0 (4+x_0^2)  \over 1+ x_0^2} ) {\ve^2 z_1 \over (\ve^2 + |z|^2 )^2} + {4 x_0^3 \over 1+ x_0^2} {{\mbox {Re} } (z^3) \over (\ve^2 + |z|^2 )^2} \\
&\quad + E_1 + B_0 [c_1 z_1 \Gamma_\ve ] \\
&\quad  +{ 8  \ve^2    \over (\ve^2 + |z|^2 )^2}   \left[  e^{-{\alpha \over 2} |\log \ve | \ell_0(z) + c_1 z_1 \Gamma_\ve } -1 \right],
\end{aligned}
$$
where $E_1$ is the explicit smooth and bounded function, given by \eqref{E1}, and $\ell_0$ is defined in \eqref{elle0}. A careful look at $B_0  (c_1 z_1 \Gamma_\ve )$ gives that
$$
B_0 (c_1 z_1 \Gamma_\ve ) = - c_1 \, x_0 \, (1+ {2\over 1+ x_0^2 } ) \Gamma_\ve (z) + E_2
$$
where $E_2$ is smooth in the variable $z $ and uniformly bounded, as $\ve \to 0$.
We conclude that, taking $c_1$ in \eqref{psi1} as defined in \eqref{defc1}
\begin{equation}\label{e0}
\begin{aligned}
 (1+ x_0^2) & S [\Psi_{1\ve}  ] (z) =  - c_1 x_0 (1+ {2\over 1+ x_0^2 } ) \Gamma_\ve (z)  + {4 x_0^3 \over 1+ x_0^2} {{\mbox {Re} } (z^3) \over (\ve^2 + |z|^2 )^2} \\
&\quad - (8 c_1 - {4 x_0 (4 + x_0^2)  \over 1+ x_0^2} ) {\ve^2 z_1 \over (\ve^2 + |z|^2 )^2} + E_1 + E_2 \\
&\quad  +{ 8  \ve^2    \over (\ve^2 + |z|^2 )^2}   \left[  e^{-{\alpha \over 2} |\log \ve | \ell_0(z )+ c_1 z_1 \Gamma_\ve } -1 \right].
\end{aligned}
\end{equation}
Observe that the first term in \eqref{e0} has size $| \log \ve|$, while the second term decays, in the expended variable $z=\ve y$, as ${1\over 1+ |y|}$. We introduce a further modification to our approximate solution $\Psi_{1\ve}$ in \eqref{psi1} to eliminate those two terms in the error.

To eliminate the first term, we let $c_2$ be the constant defined by
\begin{equation}\label{defc2}
4 c_2 = c_1 \, x_0 \, (1+ {2\over 1+ x_0^2} )
\end{equation}
so that
$$
\Delta_z \left( c_2 |z|^2 \Gamma_\ve (z) \right) - c_1 x_0 (1+ {2\over 1+ x_0^2}) \Gamma_\ve (z) = 8 c_2 \left[ {\ve^2 |z|^2 \over (\ve^2 + |z|^2 )^2} - {2 |z|^2 \over \ve^2 + |z|^2} \right].
$$
To correct the second term , we introduce $h_\ve (|z|)$ the solution to
$$
h'' +{1\over s} h' - {9 \over s^2} h + {s^3 \over (\ve^2 + s^2)^2 }=0
$$
defined by
$$
h_\ve (s ) = s^3 \int_s^1 {dx \over x^7} \int_0^x {\eta^7 \over (\ve^2 + \eta^2)^2}  \, d \eta.
$$
The function $h_\ve $ is smooth and uniformly bounded as $\ve \to 0$, and $h_\ve (s ) = O(s)$, as $s \to 0$.
Writing  $z = |z| e^{i\theta}$, we have that
\be\label{H1}
H_{1\ve}  (z ) := h_\ve (|z|) \cos 3 \theta \quad {\mbox {solves}} \quad \Delta_z \left( H_{1\ve}   \right) + {{\mbox {Re} } (z^3) \over (\ve^2 + |z|^2 )^2} = 0.
\ee
We define the following improved approximation
\begin{equation*}%\label{psi2}
\begin{aligned}
\Psi_{2\ve} (x )& = {\alpha \over 2} |\log \ve | x_0^2 - \log (1+ x_0^2) +  \Gamma_\ve (z)  \left( 1+ c_1 z_1 + c_2 |z|^2 \right)  \\
& \quad +
 {4 x_0^3 \over 1+ x_0^2} H_{1\ve} (z) ,
\end{aligned}
\end{equation*}
with $H_{1\ve}$ defined in \eqref{H1} and $c_2$ as in \eqref{defc2}. The new error function becomes
\begin{equation*}%\label{e0}
\begin{aligned}
(1+ x_0^2) \, S [\Psi_{2\ve}  ] (z) & =  E_3 - (8 c_1 - {4 x_0 (4+x_0^2)  \over 1+ x_0^2} ) {\ve^2 z_1 \over (\ve^2 + |z|^2 )^2} \\
&\quad  +{ 8  \ve^2    \over (\ve^2 + |z|^2 )^2}   \left[  e^{f_{x_0} (z)  } -1 \right],
\end{aligned}
\end{equation*}
where
\begin{equation}\label{deffx0}
\begin{aligned}
f_{x_0} (z ) &= -{\alpha \over 2} |\log \ve | \ell_0(z )+ (c_1 z_1  + c_2 |z|^2 )  \Gamma_\ve +  {4 x_0^3 \over 1+ x_0^2} H_{1\ve} (z) \\
&= (- \alpha  x_0 + 4 c_1) |\log \ve| z_1 + c_1 z_1 \log {8\over (1+ |{z \over \ve}|^2)^2} +  {4 x_0^3 \over 1+ x_0^2} H_{1\ve} (z)\\
&-{\alpha \over 2} |\log \ve| ( z_1^2 + (1+ x_0^2) z_2^2) + c_2 |z|^2 (-4|\log \ve| + \log  {8\over (1+ |{z \over \ve}|^2)^2} )
\end{aligned}
\end{equation}
with $\ell_0$ given by \eqref{elle0} and $H_{1\ve}$ in \eqref{H1}. As before, the term $E_3$ in \eqref{e0} is a smooth function, which is uniformly bounded as $\ve \to 0$.

Finally we introduce a global correction to cancel the bounded term $E_3$ in the error.
Our global final approximation is
\begin{equation}\label{psifinal}
\begin{aligned}
\Psi_{ \alpha}  (x) &=  {\alpha \over 2} |\log \ve | x_0^2 - \log (1+ x_0^2)  \\
&+ \eta_\delta (x) \left[ \Gamma_\ve (z)  \left( 1+ c_1 z_1 + c_2 |z|^2 \right) +
 {4 x_0^3 \over 1+ x_0^2} H_{1\ve} (z) \right]  + H_{2\ve}  (x)
 \end{aligned}
\end{equation}
where
\begin{equation}\label{defeta}
\eta_\delta (x ) = \eta ({|z| \over \delta} ),
\end{equation}
with $\eta $ a fixed smooth function with
\begin{equation}\label{defeta2}
\eta(s) = 1 , \quad {\mbox {for}} \quad s \leq {1 \over 2}, \quad \eta(s) = 0  , \quad {\mbox {for}} \quad s \geq 1.
\end{equation}
Let $g$ be the function with compact support defined by
$$
\begin{aligned}
g(x) &= \eta_\delta (x) E_3+ 2 \nabla \eta_\delta \nabla \left( \Gamma_\ve (z)  \left( 1+ c_1 z_1 + c_2 |z|^2 \right) +
 {4 x_0^3 \over 1+ x_0^2} H_{1\ve} (z) \right) \\
 &+ \left( \Gamma_\ve (z)  \left( 1+ c_1 z_1 + c_2 |z|^2 \right) +
 {4 x_0^3 \over 1+ x_0^2} H_{1\ve} (z) \right) \Delta \eta_\delta \\
& + B_0 \left[\eta_\delta \left( \Gamma_\ve (z)  \left( 1+ c_1 z_1 + c_2 |z|^2 \right) +
 {4 x_0^3 \over 1+ x_0^2} H_{1\ve} (z)\right) \right] \\
 &- \eta_\delta B_0 \left[ \Gamma_\ve (z)  \left( 1+ c_1 z_1 + c_2 |z|^2 \right) +
 {4 x_0^3 \over 1+ x_0^2} H_{1\ve} (z) \right].
\end{aligned}
$$
It is easy to check that
$$
\| g (x) \|_{L^\infty} \leq C_\delta
$$
for some positive constant which depends on $\delta$.
Proposition \ref{prop2} guarantees the existence of a solution to problem
$$
\begin{aligned}
\Delta H_{2\ve} + B_0 [H_{2\ve} ] + g &= 0 , \quad {\mbox {in}} \quad \R^2
\end{aligned}
$$
satisfying
\begin{equation*}
%\label{bH2}
| H_{2\ve } (x) | \leq C_\delta (1+ |x|^2).
\end{equation*}
The solution is given up to the addition of a constant.
We define the  function $H_{2\ve} (x)$ in \eqref{psifinal} to be the one which furthermore satisfies
$$
H_{2\ve} ( (x_0, 0) ) = 0.
$$
With this choice for our final approximation $\Psi_\alpha$ in \eqref{psifinal}, the error function takes the form
\begin{equation*}%\label{ef}
\begin{aligned}
&(1+ x_0^2)  S[\Psi_\alpha ] (x) = \eta_\delta \left( (1+ x_0^2 ) S[\Psi_{2\ve} ] - E_3 \right) \\
&+ (1-\eta_\delta ) {8\ve^2 \over (\ve^2 + |z|^2 )^2} e^{f_{x_0} (z)}  e^{(\eta_\delta -1) ( \Gamma_\ve (z)  \left( 1+ c_1 z_1 + c_2 |z|^2 \right) +
 {4 x_0^3 \over 1+ x_0^2} H_{1\ve} (z)) + H_{2\ve} } \\
&+\eta_\delta  {8\ve^2 \over (\ve^2 + |z|^2 )^2} e^{f_{x_0} (z) }\left( e^{(\eta_\delta -1) ( \Gamma_\ve (z)  \left( 1+ c_1 z_1 + c_2 |z|^2 \right) +
 {4 x_0^3 \over 1+ x_0^2} H_{1\ve} (z)) + H_{2\ve} } -1 \right)
\end{aligned}
\end{equation*}
where $\eta_\delta $ is given in \eqref{defeta} and $f_{x_0}$ in \eqref{deffx0}.

When $1-\eta_\delta =0$,
it is important to realize that one has
$$
\begin{aligned}
f_{x_0} (z)   &= (-\alpha x_0 + 4 c_1 ) |\log \ve | z_1 +  c_1 \log {8 \over (1+ |{z\over \ve} |^2)^2}  z_1 + O(|z|).
\end{aligned}
$$
and
$$
H_{2\ve } (z) = O(|z|).
$$
In the complementary region, where $1-\eta_\delta \not= 0$, we have that
$$
e^{f_{x_0} (z) } \leq C e^{-|z|^2},
$$
We conclude that, the error $S[\Psi_\alpha]$ of the approximate solution $\Psi_\alpha$ defined in \eqref{psifinal}, can be estimated as follows: in the region $|z| < \delta$ it has the form
\begin{equation}\label{estin}
\begin{aligned}
(1+ x_0^2) S [\Psi_\alpha ] (z)  &=  (-\alpha x_0 + 4 c_1 ) |\log \ve | {  8 \ve^2  \, z_1  \over (\ve^2 + |z|^2 )^2}  \\&
+  {  \ve^2      \over (\ve^2 + |z|^2 )^2} \,  O \left(|z| \log (2+ |{z \over \ve} |) \right) .
\end{aligned}
\end{equation}
and  in the region $|z|>\delta$,
\begin{equation}\label{estout}
|(1+ x_0^2 ) S[\Psi_\alpha ] (z) |\leq C {\ve^2 \over (\ve^2 + |z|^2)^2} e^{-|z|^2},
\end{equation}
for some constant $C>0$.

\section{The inner-outer gluing system}

We consider the approximate solution $\Psi_\alpha(x)$ we have  built in Section \ref{approx} and look  for a solution $\Psi(x) $
of the equation
\be\label{ecuacion}
S[\Psi] := L[\Psi]  + \ve^2 f( \Psi - \mu  |x|^2 )=0 \inn \R^2
\ee
where
$$
f(u)= e^u, \quad  \mu = \frac \alpha 2 |\log\ve|.
$$
We look for $\Psi$ of the form
\be \label{ansat}
\Psi (x) = \Psi_{ \alpha} (x) + \varphi(x).
\ee
Observe that, by construction, the function $\Psi_\alpha$ is symmetric with respect to $x_2$, in the sense that $\Psi_\alpha (x_1, x_2) = \Psi_\alpha (x_1 , -x_2)$. We thus also ask  $\Psi (x)$ to belong to the class of functions that symmetric with respect to $x_2$.

Here $\varphi(x)$ is a smaller perturbation of the first approximation, which we choose of the form
\be \label{ff1}
\varphi(x)=   \eta_\delta (x) \phi\left ({z \over \ve} \right)  +  \psi (x) .
\ee
We recall that
$$
%y=\ve z, \quad
z_1 = x_1- x_0 , \quad z_2 = {x_2 \over \sqrt{1+ x_0^2}}, \quad
\eta_\delta (x)  =  \eta \left ({|z| \over \delta} \right ),
$$
with $\eta$ fixed in \eqref{defeta2}.
Thus, our aim is to find $\varphi(x)$ so that
\begin{align*} % \label{pp1}
S[\Psi_\alpha + \varphi ] = {\mathcal L}_{\Psi_\alpha} [\varphi]
+N_{\Psi_\alpha} [\varphi ] + E_\alpha = 0 \quad {\mbox {in}} \quad \R^2
\end{align*}
where
$$
\begin{aligned}
E_\alpha &= S (\Psi_\alpha )\\
{\mathcal L}_{\Psi_\alpha} [\varphi] &= L [\varphi] + \ve^2 f' (\Psi_\alpha - \mu |x|^2 ) \varphi\\
N_{\Psi_\alpha } (\varphi ) &= \ve^2 \left[ f (\Psi_\alpha - \mu |x|^2 + \varphi ) -
f  (\Psi_\alpha - \mu |x|^2 ) - f' (\Psi_\alpha - \mu |x|^2 ) \varphi \right] .
\end{aligned}
$$
%We observe that
%$$
%N_{\Psi_\alpha} [\varphi] = \ve^2 f'' (\Psi_\alpha - \mu |\log \ve| |x|^2 + s \varphi ) {\varphi^2 \over 2}
%$$
%for some $s\in (0,1)$.

The following expansion holds.
$$
\begin{aligned} S( \Psi_\alpha +\vp) \ = &\
\eta_\delta\big [  L_x[\phi]  +  \ve^2 f'( \Psi_\alpha - \mu |x|^2 ) (\phi + \psi)  + E_\alpha +  N_{\Psi_\alpha } (\eta_\delta \phi + \psi ) \big]
\\
& +    L_x[\psi]   + (1-\eta_\delta )\left[    \ve^2   f'( \Psi_\alpha - \mu |x|^2 ) \psi  +   E_\alpha +  N_{\Psi_\alpha } (\eta_\delta \phi + \psi ) \right]\\
&+ L_x[\eta_\delta] \phi +  K_{ij} (x)\pp_{x_i} \eta_\delta\pp_{x_j}\phi  .
\end{aligned}
$$
Thus $\Psi$ given by \equ{ansat}-\equ{ff1} solves \equ{ecuacion} if the pair $(\phi,\psi)$ satisfies  the system of equations
\be
 L_x[\phi]  +  \ve^2 f'( \Psi_\alpha - \mu |x|^2 ) (\phi + \psi)  + E_\alpha +  N_{\Psi_\alpha } (\eta_\delta \phi + \psi ) \, =\, 0,  \quad |z|< 2\delta,
\label{in}\ee
and
\be\begin{aligned}
    L_x[\psi]  \ +  &\,   (1-\eta_\delta )  \big[
   \ve^2 f'( \Psi_\alpha - \mu |x|^2 ) \psi \, +   \, N_{\Psi_\alpha } (\eta_\delta \phi + \psi )\, +  \, E_\alpha\big ]  \\
+& \ L_x[\eta_\delta] \phi\, +\,  K_{ij} (x)\pp_{x_i} \eta_\delta\pp_{x_j}\phi \ =\  0 \inn \R^2 ,\end{aligned} \label{out}  \ee
which we respectively call  the ``inner'' and ``outer'' problems.

\medskip
Let us write \equ{in} in terms of the variable $y=\frac z\ve$. First we recall that for a function $\vp = \vp(z) $ we have
$$
L_x[\vp]   =   \frac 1{1+x_0^2} \big[ \,\Delta_z \vp +  \pp_{z_i} ( b^0_{ij} (z) \pp_{z_j} \vp)  \,\big ]
$$
for smooth functions $b^0_{ij}(z)$ with $ b^0_{ij}(0) =0$.  Hence for $\phi(y)=\varphi(\ve y) $ we have
$$
\ve^2(1+ x_0^2) L_x[\phi]  =    \Delta_y \phi  +      \pp_{y_i} ( b^0_{ij} (\ve y) \pp_{y_j} \phi)
$$
We notice that, in terms of the variable $y$ we can write
$$
\Psi_\alpha (x)  - \mu |x|^2  =  \Gamma_0(y)  - 4\log \ve -\log (1+ x_0^2)    +   |\log\ve|  \ve y_1 (-  \alpha x_0    +  4 c_1)  + \RR(y)
$$
where
$$
\Gamma_0 (y) = \log {8 \over (1+ |y|^2)^2},
$$
and $\RR(y)$ can be bounded as
\be\label{theta}    | D_y  \RR(y)|  +    |\RR(y)| \,\le \,  C \ve |y|  \log (2+|y|). \ee
At this point we make the following  choice for the parameter $\alpha$. We set
\be\label{alpha}
\alpha \ = \  \frac {4c_1} {x_0} + {\alpha_1} = {2 \over 1+ x_0^2} + \alpha_1  , \qquad |\alpha_1| \le \frac 1{|\log\ve|^{\frac 12} } .
\ee
In the expanded variable $y$ and with this choice of $\alpha$, estimates \eqref{estin} and \eqref{estout} for $E_\alpha = S[\Psi_\alpha ]$
we obtain, in the region $|y|<{\delta \over \ve}$,
\begin{equation}\label{ba1}
\tilde E_\alpha :=\ve^2(1+ x_0)^2 E_\alpha \ = \
- \frac{ 8 \alpha_1\ve y_1 }  {(1+|y|^2)^2}|\log\ve| \ + \  O \big(  \frac{\ve}{1+ |y|^{3-a}} \big )
\end{equation}
for $a \in (0,1)$.
Indeed, we first observe that
$$ \begin{aligned} \ve^4(1+ x_0)^2   f'( \Psi_\alpha (x)  - \mu |x|^2)   =  e^{\Gamma_0(y)} e^{  - \ve y_1 |\log\ve|  \alpha_1 x_0   + \RR(y)} , \end{aligned} $$
with $\RR$ as in \eqref{theta}.
For $|y|<  \frac  1{\ve|\log\ve|} $ we can expand
$$
e^{  - \ve y_1 |\log\ve|  \alpha_1 x_0  + \RR(y)} =  (1  - \ve y_1 |\log\ve|  \alpha_1 x_0    + O( \ve |y|  \log (2+|y|) +  \ve^2 \alpha_1^2|y|^2 \log^2\ve ).
$$
Hence
$$\begin{aligned}
\ve^4(1+ x_0)^2   f'( \Psi_\alpha (x)  - \mu |x|^2)  =   e^{\Gamma_0}  + b_0(y), \quad {\mbox {with}} \\
  b_0(y) =    e^{\Gamma_0} (- \ve y_1 |\log\ve|  \alpha_1 x_0)  +  O  \big ( \frac{ \ve}{1+|y|^3}  \log (2+|y|) \big ).
\end{aligned} $$
For $   \frac  1{\ve|\log\ve|} < |y| < \frac \delta\ve $ we have
$$
\ve^4(1+ x_0)^2  f'( \Psi_\alpha (x)  - \mu |x|^2) =  e^{\Gamma_0(y)} \ve^{-c\delta} =  O( \frac{\ve} { 1 +|y|^{3-a}} )
$$
for some small  $0<a<1$.  Hence we globally have
$$
\ve^4 (1+ x_0)^2  f'( \Psi_\alpha (x)  - \mu |x|^2)  = e^{\Gamma_0(y)}  +  b_0(y)
$$
where
\begin{equation}\label{defb0}
b_0(y) =  - \frac{ 8 \alpha_1\ve y_1 |\log\ve|}  {(1+|y|^2)^2} +  O \left(  \frac{\ve}{1+ |y|^{3-a}} \right ).
\end{equation}
Expansion \eqref{ba1} readily follows.
 Similarly, for $b_0$ of this type we get the expansion
\begin{equation}\label{house2}
\mathcal N (\vp) := \ve^2(1+ x_0)^2 N_{\Psi_\alpha} (\vp)  =  ( e^{\Gamma_0(y)}   +   b_0)  ( e^\vp -1 -\vp)
\end{equation}
%and
%$$
%\ttt  E_\alpha  =\ve^2(1+ x_0)^2 E_\alpha \ = \
%- \frac{ 8 \alpha_1\ve y_1 }  {(1+|y|^2)^2}|\log\ve| \ + \  O \big(  \frac{\ve}{1+ |y|^{3-a}} \big )
%$$
Then, the inner problem \equ{in} becomes
\be\label{inner}
\Delta_y \phi  +    f'(\Gamma_0) \phi +   B[\phi]  +   \mathcal N (\psi + \eta_\delta \phi ) +   \ttt  E_\alpha +  (f'(\Gamma_0) + b_0 ) \psi   = 0 \inn B_R
\ee
where $R= \frac {2\delta}{\ve} $ and
\begin{equation}\label{defB}
B[\phi] = \pp_{y_i} (b_{ij}^0(\ve y) \pp_j\phi)  + b_0(y) \phi .
\end{equation}
The idea is to solve this equation, coupled with the outer problem \equ{out} in such a way that  $\phi$ has the size of the error $\ttt E_\alpha$ with two powers less of decay in $y$, say
$$
(1+|y|) |D_y\phi(y)|+  |\phi(y)| \le  \frac {C \ve|\log\ve| } {1+ |y|^{1-a}}.
$$
We write the outer problem as
\be\label{outer}
    L_x[\psi]  \ +  \,   G(\psi, \phi,\alpha)  \ =\  0 \inn \R^2
\ee
where

\be
G(\psi, \phi,\alpha) =
   V(x)\psi \, +   \, N(\eta_\delta \phi + \psi ) \, +  \, E^o (x)\,
+ \ A[\phi],
\label{GG}\ee
with
$$\begin{aligned}
V(x)\ = &\ (1-\eta_\delta ) \ve^2 f'( \Psi_\alpha - \mu |x|^2 ) , \\  N(\vp) \  = &\  (1-\eta_\delta )  N_{\Psi_\alpha } (\vp )\\
E^o (x)\ = &\  (1-\eta_\delta) E_\alpha \\  A[\phi]\ =&\  L_x[\eta_\delta] \phi\, +\,  K_{ij} (x)\pp_{x_i} \eta_\delta\pp_{x_j}\phi.
\end{aligned}$$
We observe that for $|z|> \frac \delta 2 $ we have
$$
\Psi_\alpha(x) -  \frac\alpha 2 |\log\ve||x|^2 =  O(1)  + \frac\alpha 2 |\log\ve||x_0|^2  - \frac\alpha 2 |\log\ve||x|^2
$$
An important point is that from the choice of $\alpha$ in \equ{alpha} we have that
$$
\frac\alpha 2|x_0|^2  < \frac {x_0^2+\sigma'}{1+x_0^2},
$$
for some $\sigma'>0$ arbitrarily small.
It follows that there exists $0<b<1$, $b={1-\sigma' \over 1+ x_0^2}$  which depends on $x_0$,
$$
\ve^2\exp( \Psi_\alpha(x) -  \frac\alpha 2 |\log\ve||x|^2 )  =  O(\ve^{1+b  } e^{-|\log\ve| |x|^2} ).
$$
Hence the following bounds hold.
\begin{equation}\label{house1}
\begin{aligned} |V(x)|\ \le&\   \ve^{1+b} e^{ -  |x|^2}  , \\
|N(\eta_\delta\phi+ \psi)|  \ \le &\   \ve^{1+b}  e^{ - |x|^2} (|\phi|^2\eta_\delta^2 +  |\psi|^2) \\
|E_\alpha^o (x)|\ \le &\   \ve^{1+b}  e^{ -  |x|^2}.
\end{aligned}
\end{equation}

In the next two sections we will establish linear results that are the basic tools to solve system \equ{inner}-\equ{outer} by means of a
fixed point scheme.

\section{Linerized inner problem}

In this section we consider the problem
\be\label{00}
\begin{aligned}
\Delta \phi + e^{\Gamma_0(y)}\phi  +  h(y)= 0 \inn \R^2.  %\phi(y) \to 0 \ass  |y|\to +\infty .
\end{aligned}
\ee
We want to solve this problem  in a topology of decaying functions in H\"older sense. For numbers $m>2$, $0<\alpha<1$ we consider the following norms
\be\label{norma} \begin{aligned}
\|h\|_{m}  = & \sup_{y\in \R^n} (1+|y|)^m|h(y)|, \\
\|h\|_{m,\alpha}  = &\|h\|_{m}  +   (1+|y|)^{m + \alpha}[h]_{B_1(y),\alpha} ,
\end{aligned}
\ee
where we use the standard notation
$$
[h]_{A,\alpha} = \sup_{z_1,z_2\in A}  \frac {|h(z_1) -h(z_2)| } {|z_1-z_2|^\alpha}
$$
We consider the functions $Z_i(y)$ defined as
$$ Z_i(y)  =  \pp_{y_i}  \Gamma_0(y) =   -4 \frac {y_i} {|y|^2+1}, \  i=1,2, \quad  Z_0(y) =    2+  y\cdot \nn \Gamma_0(y) =   2 \frac {1 -2|y|^2} {|y|^2+1}  .     $$

\begin{lemma}\label{lemat}
Given $m>2$ and $0<\alpha < 1$,
there exists a $C>0$ and a solution $\phi =  \TT [ h]$ of problem $\equ{00}$ for each $h$ with $\|h\|_m <+\infty$ that defines a linear operator of $h$ and satisfies the estimate
\be\label{cota} \begin{aligned}
& (1+|y|) | \nn \phi (y)|  +  | \phi (y)| \\  &
\,  \le  \,  C \big [ \,  \log (2+|y|) \,\big|\int_{\R^2} h Z_0\big|  +    (1+|y|) \sum_{i=1}^2 \big|\int_{\R^2} h Z_i\big|  +  (1+|y|)^{2-m} \|h\|_{m}   \,\big ]. \end{aligned}
\ee
In addition, if
$\|h\|_{m,\alpha} <+\infty$, we have
\be\label{cotaa} \begin{aligned}
& (1+|y|^{2+\alpha})  [D^2_y \phi]_{B_1(y),\alpha}  +(1+|y|^2)  |D^2_y \phi (y)| \\  &
\,  \le  \,  C \big [ \,  \log (2+|y|) \,\big|\int_{\R^2} h Z_0\big|  +    (1+|y|) \sum_{i=1}^2 \big|\int_{\R^2} h Z_i\big|  +  (1+|y|)^{2-m} \|h\|_{m,\alpha}   \,\big ]. \end{aligned}
\ee

\end{lemma}

\proof
The existence of this inverse is essentially known, we provide a proof. We assume that $h$ is complex-valued. Setting $y=re^{i\theta}$.
We write
$$ h(y) =  \sum_{k=-\infty}^{\infty}  h_k (r) e^{ik\theta} , \quad \phi(y)= \sum_{k=-\infty}^{\infty} \phi_k (r) e^{ik\theta}
$$
The equation is equivalent to
\be
\LL_k [\phi_k] + h_k(r) = 0 , \quad r\in(0,\infty)
\label{01}\ee
where
$$
\LL_k [\phi_k] =   \phi_k''  + \frac 1r \phi_k'  +  e^{\Gamma_0}  \phi_k   - \frac{k^2}{r^2} \phi_k
$$
Using the formula of variation of parameters,  the following formula (continuously extended to $r= \frac 1{\sqrt{2}}$ defines a smooth solution of \equ{01} for $k=0$:
%$$
%\phi_0  (r) =   \ttt Z(r)  \int_0^r h(s) Z(s) s\,ds  +   Z(r) \int_r^\infty h(s) \ttt Z(s) s\,ds
%$$
$$
\phi_0  (r) =     -z(r)\int_{\frac 1{\sqrt{2}}}^r  \frac {ds}{ s z(s)^2}  \int_0^s h_0(\rho) z(\rho) \rho\,ds, \quad z(r) = \frac{2r^2 - 1}{ 1 + r^2}
$$
Noting that $\int_0^\infty h_0(\rho) z(\rho) \rho\,ds = \frac 1{2\pi}\int_{\R^2} h(y)Z_0(y)\, dy      $
we see that this function satisfies
$$
|\phi_0(r)| \, \le \,   C\big[ \, \log (2 + r) \Big| \int_{\R^2} h(y)Z_0(y)\, dy      \Big|   \, +\,  (1+r)^{2-m} \|h\|_{ m} \big ].
$$
Now we observe that
$$
\phi_k  (r) =     -z(r)\int_0^r  \frac {ds}{ s z(s)^2}  \int_0^s h_k(\rho) z(\rho) \rho\,ds, \quad z(r) = \frac{4r}{ 1 + r^2}
$$
solves \equ{01} for $k =-1,1$ and satisfies
$$
|\phi_k(r)| \, \le \,   C\big[ \,  (1+ r) \sum_{j=1}^2 \Big| \int_{\R^2} h(y)Z_j(y)\, dy      \Big|   \, +\,  (1+r)^{2-m} \|h\|_{ m} \big ].
$$
For $k=2$ there is a function $z(r)$ such that $\mathcal L_2[z] = 0 $,  $z(r)\sim r^{2}$ as $r\to 0$ and as $r\to \infty$.  For $|k|\ge 2$ we have that
$$
\bar\phi_k  (r) =     \frac 4{ k^2} z(r)\int_0^r  \frac {ds}{ s z(s)^2}  \int_0^s |h_k(\rho)| z(\rho) \rho\,ds,
$$
is a positive supersolution for equation \equ{01}, hence the equation has a unique solution $\phi_k$  with $|\phi_k(r)| \le \bar\phi_k  (r)$. Thus
$$
|\phi_k(r)| \, \le \,    \frac C {k^2}   (1+r)^{2-m} \|h\|_{ m}, \quad |k|\ge 2.
$$
Thus, for the functions $\phi_k$ defined  $$\phi(y)= \sum_{k=-\infty}^{\infty} \phi_k (r) e^{ik\theta} $$  defines a linear operator of functions $h$ which is a solution of equation \equ{00} which, adding up the individual estimates above,  it satisfies the estimate
\be\label{cot} \begin{aligned}
 | \phi (y)|
\,  \le  \,  C \big [ \,  \log (2+|y|) \,\big|\int_{\R^2} h Z_0\big|  +    (1+|y|) \sum_{i=1}^2 \big|\int_{\R^2} h Z_i\big|  +  (1+|y|)^{2-m} \|h\|_{m}   \,\big ]. \end{aligned}
\ee
As a corollary we find that similar bounds are obtained for first and second derivatives. In fact, let us set
for a large $y = R e$, $R=|y|\gg 1$,
$
\phi_R(z)  =  {R^{m-2}} \phi (R(e+z))    .
$
Then we find
$$
\Delta_z\phi_R  +   R^{-2} W(R(e+ z))\phi_R +  h_R (z)   = 0, \quad |z| <\frac 12
$$
where $h_R(z) = R^{-m}h (R(e+z)) $.
Let us set,
$$
\delta_i =  \Big| \int_{\R^2}  h Z_i \Big | , \quad i=0,1,2.
$$
Then  from \equ{cot}, and a standard elliptic estimate we find
 $$\|\nn_z \phi_R \|_{L^\infty( B_{\frac 14}(0) )} + \|\phi_R \|_{L^\infty( B_{\frac 12}(0) )} \, \le\,  C\Big [ \delta_0 \log R +   \sum_{i=1}^2\delta_i R +   \| h\|_{m} \Big ].$$
  Clearly we have
$$
   \|h_R \|_{L^\infty(B_{\frac 12} (0))} \ \le\  C\| h\|_{m}, \quad    [h_R ]_{B_{\frac 12} (0)} \ \le\  C\| h\|_{m,\alpha}  .
$$
From interior Schauder estimates and the bound for $\phi_R$  we then  find
$$
\|D^2_{z} \phi_R \|_{L^\infty( B_{\frac 14}(0) ) } +  [ D^2_{z} \phi_R ]_{B_{\frac 14}(0), \alpha} \, \le\,  C\Big [ \delta_0 \log R +   \sum_{i=1}^2\delta_i R +   \| h\|_{m,\alpha} \Big ].
$$
From these relations, estimates \equ{cota} and \equ{cotaa} follow. \qed

\bigskip
We consider now the problem for a fixed number $\delta>0$ and a sufficiently large $R>0$ we consider the equation
\be
\Delta \phi  +  e^{\Gamma_0} \phi  + B[\phi]  + h(y)  = \sum_{j=0}^2  c_i e^{\Gamma_0} Z_i \inn B_R
\label{eee} \ee
where $c_i$ are real numbers,
\be\label{B}
B[\phi]=   \pp_{i} ( b_{ij} (y)\pp_j \phi)  + b_0(y)\phi
\ee
and the coefficients satisfy the bounds
\be\label{coef} \begin{aligned}
 & \,  (1+|y|)^{-1} |b_{ij}(y)|  +   | D_yb_{ij}(y)|   \\     +  &  \      (1+|y|)^2 |b_{0}(y)| + (1+|y|)^3| D_yb_{0}(y)| \\
 \, \le\,  & \,\, \delta R^{-1} \inn B_R\, . \end{aligned}   \ee
For a function $h$ defined in $A\subset \R^2$ we denote by $\|h\|_{m,\alpha,A}$ the numbers defined in \equ{norma}
 but with the sup  taken with elements in $A$ only, namely
  $$
  \|h \|_{m, A}=  \sup_{y\in A} | (1+|y|)^m h(y)|] ,  \quad   \|h \|_{m,\alpha, A} =  \sup_{y\in A} (1+|y|)^{m+\alpha}[ h]_{B(y,1)\cap A}   +     \|h \|_{m, A}
 $$
 Let us also define, for a function of class $C^{2,\alpha}(A)$,
 \be\label{star}
 \| \phi\|_{*,m-2, A} =  \|D^2\phi\|_{m,\alpha, A}   +  \|D\phi\|_{m-1,A} + \|D\phi\|_{m-2,A} .
 \ee
 In this notation we omit the dependence on $A$ when $A= \R^2$. The following is the main result of this section.

\begin{prop}\label{prop1}  There are numbers $\delta,C>0$ such that for all sufficiently large $R$ and any differential operator $B$ as in $\equ{B}$ with bounds $\equ{coef}$,
Problem $\equ{eee}$ has a solution $\phi = T[h]$ for certain scalars $c_i= c_i[h]$, that defines a linear operator of $h$ and satisfies
\begin{align*} % \label{cota3}
\| \phi\|_{*,m-2, B_R} \ \le\ C \|h \|_{m,\alpha, B_R} .
\end{align*}
In addition, the linear functionals $c_i$ can be estimated as
\begin{align*} % \label{estc}\begin{aligned}
c_0[h]\, = & \, \gamma_0\int_{B_{R}}  h Z_0  + O( R^{2-m})  \|h \|_{m,\alpha, B_R} , \\    c_i[h]\, = &\, \gamma_i\int_{B_{R}}  h Z_i  + O( R^{1-m})  \|h \|_{m,\alpha, B_R}, \ i=1,2.
\end{align*}
where
$\gamma_i^{-1} = \int_{\R^2} e^{\Gamma_0} Z_i^2$, $i=0,1,2$.
\end{prop}

\proof
We consider a standard linear extension operator $h\mapsto \ttt h $ to entire $\R^2$,
in such a way that the support of $\ttt h$ is contained in $B_{2R}$ and $\|\ttt h\|_{m,\alpha} \le C\|h\|_{m,\alpha, B_R}$ with $C$ independent of all large $R$. In a similar way,
we assume with no loss of generality that the coefficients of $B$ are of class $C^1$ in entire $\R^2$,  have compact support in $B_{2R}$ and globally satisfy bounds \equ{coef}.
Then we consider the auxiliary problem in entire space

\be
\Delta \phi  +  e^{\Gamma_0} \phi  + B[\phi]  + \ttt h(y)  = \sum_{j=0}^2  c_i e^{\Gamma_0} Z_i \inn \R^2
\label{001} \ee
where, assuming that $\|h\|_m<+\infty$ and $\phi$ is of class $C^2$,
 $c_i = c_i[h,\phi]$ are the scalars defined as so that
$$
\gamma_i \int_{\R^2} (B[\phi]  + \ttt h(y))Z_i  =    c_i  ,  \quad \gamma_i^{-1} = \int_{\R^2}  e^{\Gamma_0} Z_i^2
$$
For a function $\phi$  of class $C^{2,\alpha} (\R^2)$ we define
 $$
 \|\phi\|_{*, m-2,\alpha} =  \|  D^2_y\phi \|_{m,\alpha}   + \|  D_y\phi \|_{m-1}+ \|\phi \|_{m-2}
  $$
Since $B[Z_i]=  O((1+|y|)^{-2}) $ and $ \int_{\R^2} \frac {dy}{1+|y|^{m} } <+\infty$  we get
 $$
  \int_{\R^2} B[\phi]Z_i =   \int_{\R^2} \phi B[Z_i]  = O(\|\phi\|_{m-2}) R^{-1}   .
 $$
 In addition, we readily check that
 $$ \| B[\phi] \|_{m,\alpha}  \le  C\delta  \|\phi\|_{*, m-2,\alpha}
. $$
Let us consider the Banach space $X$ of all $C^{2,\alpha}(\R^2)$
  functions with
  $
  \|\phi\|_{*, m-2,\alpha} <+\infty.
 $
 We find a solution of \equ{001}
if we solve the equation
\be\label{fp}
 \phi  =   \mathcal A  [\phi]  +  \mathcal H ,\quad \phi \in X
\ee
where
$$
\mathcal A  [\phi]
 =  \mathcal T\Big [B[\phi] -\sum_{i=0}^2 c_i[0, \phi] e^{\Gamma_0}Z_i  \Big] ,\quad
 \mathcal H  = \mathcal T\Big [  \ttt h   - \sum_{i=0}^2 c_i[\ttt h,0] e^{\Gamma_0}Z_i  \Big] .
$$
and $\mathcal T$ is the operator built in Lemma  \ref{lemat}.
We observe that
$$
\|\mathcal A  [\phi] \|_{*, m-2,\alpha} \le C\delta \|\phi\| _{*, m-2,\alpha}, \quad \|\mathcal H \|_{*, m-2,\alpha} \le C \|h \|_{ m,\alpha, B_R}.
$$
Fixing $\delta$ so that $\delta C<1$, we find that Equation \equ{fp} has a unique solution, that defines a linear operator of $h$, and satisfies
$$
\|\phi \|_{*, m-2,\alpha} \ \le\  C \|h \|_{ m,\alpha, B_R}
$$
The result of the proposition follows by just setting $T[h] = \phi\big|_{B_R}$. The proof is concluded. \qed

\section{The linear outer problem}

Let us recall $L$ defined by \eqref{defL}
\begin{equation*}%\label{defL1}
L[\psi] = {\mbox {div}} (K \nabla \psi)
\end{equation*}
where
$K = K(x_1 , x_2)$ is the matrix
\begin{equation*}%\label{defK}
K(x_1 , x_2 ) = \frac{1}{1+x_1^2+x_2^2}
\left(
\begin{matrix}
1+x_2^2 & -x_1  x_2\\
-x_1 x_2 & 1+x_1^2
\end{matrix}
\right) ,
\end{equation*}
that is, the same matrix \eqref{defK-0} with $h=1$.

In this section we consider the Poisson equation for the operator $L$
\be
\label{louter}
L[\psi]  +  g(x) = 0  \inn \R^2 ,
\ee
for a bounded function $g$.

It is sufficient  to restrict our attention to the case of functions  $g(x)$ that satisfy the decay condition
$$
\| g\|_\nu\, :=\, \sup_{x\in \R^2} (1+|x|)^\nu|g(x)|\, <\,+\infty \, ,
$$
where $\nu >2$. We have the validity of the following result
\begin{prop}\label{prop2}
There exists a solution $\psi(x)$ to problem $\equ{louter}$, which is of class $C^{1,\alpha}(\R^2)$ for any $0<\alpha<1$,  that defines a linear operator $\psi = {\mathcal T}^o (g) $ of $g$
and satisfies the bound
\be\label{estimate}
|\psi(x)| \,\le \,     C{ \| g\|_\nu} (1+ |x|^2),
\ee
for some positive constant $C$.
\end{prop}

\proof
To solve Equation \equ{louter}
we decompose $g$ and $\psi$ into Fourier modes as
$$ g(x) = \sum_{j=-\infty}^{\infty} g_j(r) e^{ji\theta}, \quad    \psi(x) = \sum_{j=-\infty}^{\infty}  \psi_j (r) e^{ji\theta},\quad x=re^{i\theta} .$$
It is useful to derive an expression of the operator $L[\psi]$ expressed in polar coordinates. With some abuse of notation we write
$\psi(r,\theta) = \psi(x)$ with $x= r \, e^{i\theta}$. The following expression holds:
\be
L[\psi]  =  {1\over 1 + r^2} \left( {1 \over r^2} +1 \right) \pp_\theta^2 \psi + {1 \over r} \pp_r \left( {r \over 1 + r^2} \pp_r\psi \right).
\label{polar}\ee
To prove this, recall
that for a vector field $\vec F = F_r \hat r + F_\theta \hat\theta$
where $$\hat r = {1\over r} (x_1,x_2), \quad \hat \theta = {1\over r} (-x_2,x_1), $$ we have
$$
{\mbox {div}}(\vec F)\ =\  {1\over r} \pp_r ( r F^r) + {1\over r} \pp_\theta F^\theta, \quad  \nn \psi =  \psi_r \hat r + \frac {  \psi_\theta}r\hat \theta.  $$
Then
$$
\begin{aligned}
L[\psi]\  &=\   {\mbox {div}} (K \nabla \psi ) = {\mbox {div}}\left(  \frac{1}{1+x_1^2+y_1^2}
\left(
\begin{matrix}
\kappa^2+x_2^2 & -x_1x_2 \\
-x_1x_2 & \kappa^2+x_1^2
\end{matrix}
\right) \nabla \psi \right)   \\
&=  {\mbox {div}}\left( \frac{1}{1+x_1^2+x_2^2}
\left(
\begin{matrix}
x_2^2 & -x_1x_2 \\
-x_1x_2 & x_1^2
\end{matrix}
\right) \nabla \psi \right)  + {\mbox {div}} \left( \frac{1}{1+ x_1^2+x_2^2}
\nabla \psi \right)\\
&= L_1 [\psi] + L_2 [\psi].
\end{aligned}
$$
Using the polar coordinates formalism, we get
$$
\begin{aligned}
L_1 [\psi ] &= {\mbox {div}}\left( {r^2 \over 1 + r^2}
\left[ \begin{matrix} \sin^2 \theta & -\sin \theta \cos \theta \\
-\sin \theta \cos \theta & \cos^2 \theta \end{matrix} \right] (\psi_r \hat r + {\psi_\theta \over r} \hat \theta ) \right) \\
&=  {\mbox {div}}\left( {r^2 \over 1 + r^2}
\left[ \begin{matrix} -\sin \theta \,  \hat \theta^T \\
 \cos \theta \, \hat \theta^T \end{matrix} \right] (\psi_r \hat r + {\psi_\theta \over r} \hat \theta ) \right) \\
&=\, {\mbox {div}} ( \psi_\theta {r \over 1 + r^2} \hat \theta )\  =\  {1\over 1 + r^2} \psi_{\theta\theta} .
\end{aligned}
$$
Similarly,
$$
L_2 [\psi ] = {1\over r} \left( { r \over r^2 + 1} \psi_r \right)_r + {1 \over 1 + r^2 } {\psi_{\theta \theta} \over r^2}.
$$
Combining the above relations, expression \equ{polar}  follows. A nice feature of this operator is that it decouples the Fourier modes. In fact, equation
\equ{louter}  becomes  equivalent to the following infinite set of ODEs:
\be\label{LLk} \begin{aligned}
&L_k[\psi_k] +  g_k(r) = 0 , \quad r\in (0,\infty),
\\
&L_k[\psi_k]:=
{1\over r} \big( {r \over r^2 + 1} \psi_k'\big)'  -  {k^2\over 1 + r^2} \left( {1 \over r^2} +1 \right)\psi_k   , \quad k\in \Z.
 \end{aligned}  \ee
 The operator $L_k $ when $r\to 0$ or $r\to +\infty$ resembles
 $$\begin{aligned}
 L_k[p]\ \sim &\ \frac 1r(rp') - \frac {k^2p} {r^2} \ass r\to 0 \\
  L_k[p]\ \sim &\ \frac 1r(\frac 1r p') - \frac {k^2p} {r^2} \ass r\to +\infty
  \end{aligned}
 $$
 For $k\ge 1$, thanks to the maximum principle,   this implies that the existence of a positive function $z_k(r)$ with $L_k[z_k]=0$ with
  $$\begin{aligned}
z_k(r) \sim  r^k \ass r\to 0 \\
z_k(r)\sim   r^{\frac 12 } e^{kr} \ass r\to +\infty.
  \end{aligned}
 $$
 Let us consider the equation (for a barrier) for $k=1$,
 $$
 L_1[\bar \psi]  + \frac 1{1+ r^\nu } = 0
 $$
  A solution of this equation is given by
 $$
 \bar \psi(r) =   z_1(r) \int_r^\infty \frac {ds}{ s z_1(s)^2} \int_0^s \frac 1{1+ \rho^\nu } z_1(\rho) \rho\, d\rho  .
 $$
 and we have
   $$\begin{aligned}
\bar \psi(r) =O(  r^2 ) \ass r\to 0 \\
\bar \psi(r) = O( r^{\nu-2}) \ass r\to +\infty.
  \end{aligned}
 $$
We also observe that for  $k\ge 2$ this function works as a barrier. In fact we have
 $$
 |g_k(r)|  \le \frac {\|g\|_{\nu}} {1+ r^\nu  }
 $$
 hence \equ{LLk} for $|k|\ge 2$  admits as a positive barrier. Using the maximum principle,
 the function
 $$
 \psi_k(r) =  z_k(r) \int_r^\infty \frac {ds}{ s z_k(s)^2} \int_0^s h_k(\rho) z_k(\rho) \rho\, d\rho  ,
 $$
which is the unique decaying solution   \equ{LLk},  satisfies the estimate
 $$
 |\psi_k(r)| \le  \frac 4{k^2}  {\|g\|_{\nu}} \bar \psi(r) .
 $$
 Finally, let us consider the case $k=0$. In that case, the following explicit formula yields a solution
 $$
 \psi_0(r)  = -\int_0^r  \frac{1 + s^2}{h^2  s}
 \, ds \int_0^s  g_0(\rho) \rho \, d\rho  .
 $$
 In this case we  get the bound
 $$
 |\psi_0(r)| \le      C\|g\|_{\nu} (1+  r^2) .
 $$
 %We conclude that there is a unique solution $\psi$ with $\psi(0) =0$, given by
 The function
 $$
 \psi(x) :=  \sum_{j=-\infty}^{\infty}  \psi_j (r) e^{ji\theta},
 $$
 with the $\psi_k$ being the functions built above, clearly  defines a linear operator of $g$  and satisfies estimate \equ{estimate}.
 The proof is concluded. \qed

\section{Solving the inner-outer gluing system}

In this section we will formulate
System \equ{inner}-\equ{outer} (for an appropriate value of the parameter $\alpha$ in \equ{alpha})
in a fixed point formulation that involves the linear operators defined in the previous sections.

 We let $X^o$ be the Banach space of all functions $\psi \in C^{2,\alpha}(\R^2)$ such that
$$
\|\psi \| < +\infty,
$$
and formulate the outer equation \equ{outer} as the fixed point problem in $X^o$,
\begin{align*}%\label{outer2}
\psi  =  \TT^o [   G(\psi, \phi,\alpha)   ], \quad \psi\in X^o
\end{align*}
where ${\mathcal T}^o$ is defined in Proposition \ref{prop2}, while  $G$ is the operator given by \equ{GG}. We decompose the inner problem
\equ{inner} as follows: consider
\begin{align*} % \label{inner10}
\Delta_y \phi  +    f'(\Gamma_0) \phi +   B[\phi]  +  H(\phi, \psi,\alpha)     = 0 \inn B_R
\end{align*}
where $R={2\delta \over \ve}$ and
$$
H(\phi, \psi,\alpha)  = {\mathcal N} (\psi + \eta_\delta \phi ) +   \ttt  E_\alpha +  (f'(\Gamma_0) + b_0 ) \psi
$$
Let $X_*$ be the Banach space of functions $\phi \in C^{2, \alpha} (B_R)$ such that
$$
\| \phi \|_{*,m-2,B_R} <\infty
$$
(see \eqref{star}).
We introduce constants $c_i$
\be\label{inner1}
\Delta_y \phi_1  +    f'(\Gamma_0) \phi_1 +   B[\phi_1]  + B[\phi_2]+ H(\phi_1+\phi_2, \psi,\alpha)  -\sum_{i=0}^2 c_i e^{\Gamma_0} Z_i   = 0 \inn B_R
\ee
and formulate this problem using the operator $T$ in Proposition \ref{prop1}, with $c_i = c_i [ H(\phi_1+\phi_2, \psi,\alpha) +B(\phi_2) ]$, $i=0,1,2$, and
$\phi_1 \in X_*$
$$
\phi_1 =  T(  H(\phi_1+\phi_2, \psi,\alpha) + B[\phi_2] )
$$
and we require that $\phi_2$ solves
$$
\Delta_y \phi_2  +    f'(\Gamma_0) \phi_2 +  c_0 e^{\Gamma_0} Z_0  = 0 \inn \R^2 .
$$
We solve Problem \eqref{inner2} by using the operator $\mathcal{T}$ in Lemma \ref{lemat}. We write
\be\label{inner2}
\phi_2 =  \TT [  c_0[ H(\phi_1+\phi_2, \psi,\alpha) +B(\phi_2) ]e^{\Gamma_0} Z_0   ].
\ee
Having in mind the a-priori bound in \eqref{cota}, \eqref{cotaa} in Lemma \ref{lemat}, it is natural  to ask that $\phi_2 \in C^{2,\alpha}$,
\begin{align*}
 (1+|y|^{2+\alpha})  [D^2_y \phi]_{B_1(y),\alpha} & +(1+|y|^2)  |D^2_y \phi (y)|  +(1+|y|) | \nn \phi (y)|  +  | \phi (y)|
 \\
 & \leq C  \log (1+ |y|) .
\end{align*}
We call $\| \phi \|_{**,m-2,\alpha}$ the infimum of the constants $C$ that satisfy the above inequality, and denote by $X_{**}$ the Banach space of functions $\phi \in C^{2,\alpha}$ with $\| \phi \|_{**, m-2 , \alpha} <\infty$.

We couple equations \eqref{outer}, \eqref{inner1} and \eqref{inner2} with the relations
$$
c_1[ H(\phi_1+\phi_2, \psi,\alpha) +B(\phi_2) ]=0 , \quad c_2[ H(\phi_1+\phi_2, \psi,\alpha) +B(\phi_2) ]=0.
$$
The second equation  $c_2= 0$ is automatically satisfied thanks to the fact that all functions are even in the variable $y_2$ (or equivalently in the variable $z_2$, and in $x_2$).
The first equation $c_1=0$ becomes an equation in the parameter $\alpha_1$. Recall that $\alpha_1$ is chosen in \eqref{alpha} as
$$
\alpha \ = \  \frac {4c_1} {x_0} + {\alpha_1} = {2 \over 1+ x_0^2} + \alpha_1  , \qquad |\alpha_1| \le \frac 1{|\log\ve|^{\frac 12} } .
$$
Equation $c_1=0$ becomes
\begin{equation}\label{c1}
\begin{aligned}
 \alpha_1 \ve |\log \ve | \beta = & F(\psi, \phi_1 , \phi_2 , \alpha_1) , \quad {\mbox {where}}\\
 F(\psi, \phi_1 , \phi_2 , \alpha_1) & =\, \gamma_1\int_{B_{R}}  \left[ {\mathcal N} (\psi + \eta_\delta \phi ) +     (f'(\Gamma_0) + b_0 ) \psi +B(\phi_2) \right]  Z_i  \\
 &+ O( R^{1-m})  \| {\mathcal N} (\psi + \eta_\delta \phi ) +    \ttt E_\alpha + (f'(\Gamma_0) + b_0 ) \psi +B(\phi_2) \|_{m,\alpha, B_R}, \
\end{aligned}
\end{equation}
and $\beta$
$$
\beta= -8 \gamma_1 \left( \int_{B_R}  \frac{  y_1 }  {(1+|y|^2)^2} Z_1 \ + \  O \big(  {1\over |\log \ve|^{1\over 2}} \big) \right).
$$
The final step to conclude the proof of our result is to find $\psi$, $\phi_1$, $\phi_2$ and $\alpha_1$ solution of the fixed point problem
\be \label{fix}
(\psi, \phi_1 , \phi_2 , \alpha_1) = {\mathcal A} (\psi, \phi_1 , \phi_2 , \alpha_1)
\ee
given by
\begin{equation}\label{fixp}
\begin{aligned}
\psi  &=  \TT^o [   G(\psi, \phi_1+ \phi_2 ,\alpha)   ], \quad \psi\in X^o\\
\phi_1 &=  T[  H(\phi_1+\phi_2, \psi,\alpha) +B(\phi_2)]\\
\phi_2 &=  \TT [  c_0[ H(\phi_1+\phi_2, \psi,\alpha) +B(\phi_2) ]e^{\Gamma_0} Z_0   ]\\
 \alpha_1  &=  {1\over \ve |\log \ve | \beta} \, F(\psi, \phi_1 , \phi_2 , \alpha_1).
\end{aligned}
\end{equation}
Let $b\in (0,1)$ the number introduced in \eqref{house1}, $m>2$ and define
\be \label{ball}
\begin{aligned}
B_M&=\{ (\psi, \phi_1 , \phi_2 , \alpha_1 ) \in X^o\times X_*\times X_{**} \times \R \, : \,  \\
& \| \psi \|_\infty \leq M \ve^{1+b} , \, \| \phi_1\|_{*, m-2, B_R} \leq M \ve |\log \ve |^{1\over 2} , \\
&\, \| \phi_2\|_{**, m-2, \alpha} \leq M \ve |\log \ve |^{1\over 2} , \, |\alpha_1 | \leq {M \over |\log \ve |^{1\over 2}} \},
\end{aligned}
\ee
for some positive constant $M$ independent of $\ve$.
We shall solve \eqref{fix}-\eqref{fixp} in $B$.

We first show that ${\mathcal A} (B) \subset B$. Assume that  $(\psi, \phi_1 , \phi_2 , \alpha_1 ) \in B$. We first want to show that ${\mathcal A} (\psi, \phi_1 , \phi_2 , \alpha_1 ) \in B$.
From \eqref{GG} and \eqref{house1}, we get that
\begin{align*}
|G(\psi , \phi_1 + \phi_2 , \alpha_1 )| &\leq {C \over 1+ |x|^\nu}  \ve^{1+b} (1+ |\phi_1+\phi_2|^2 \eta_\delta + |\psi|^2 ) \\
&+ {C \over 1+ |x|^\nu}
\ve^{2-m} (\|\phi_1\|_{*, m-2, B_R} + \| \phi_2 \|_{**, m-2, \alpha} )
\end{align*}
for some $\nu>2$. From Proposition \ref{prop2}, we get that
\be \label{f1}
\begin{aligned}
\| \psi \|_\infty &= \| {\mathcal T}^o ( G(\psi , \phi_1 + \phi_2 , \alpha_1 ) ) \|_\infty \leq C \ve^{1+b}.
\end{aligned}
\ee
From \eqref{house2}, \eqref{ba1}, \eqref{defb0}, we get, for some $a\in (0,1)$,
\begin{align*}
|H(\phi_1 + \phi_2 , \psi , \alpha_1) | &\leq C ( {\alpha_1 \ve |\log \ve | \over (1+ |y|^2)^2} + {\ve \over (1+|y| )^{3-a}} ) \\
&+ {C \over 1+ |y|^{3-a}} ({1\over (1+ |y|)^{1+a} } + {\alpha_1 \ve |\log \ve | \over (1+ |y|)^a} + \ve ) |\psi| \\
&+ {C \over (1+ |y|)^4} (|\psi|^2 + |\eta_\delta \phi_1|^2 + |\eta_\delta \phi_2|^2 ).
\end{align*}
Using the assumptions on $\psi$, $\phi_1$ and $\phi_2$, we get that
\be \label{f2}
\|H(\phi_1 + \phi_2 , \psi , \alpha_1) \|_{m, \alpha, B_R} \leq C \ve |\log \ve|^{1\over 2} .
\ee
From \eqref{defB}, we get
\begin{align*}
| B[\phi_2] | &\leq C {\ve^2 |\log \ve | \over 1+ |y|} \| \phi_2 \|_{**, m-2, \alpha},
\end{align*}
and
\be \label{f3}
\|B(\phi_2)  \|_{m, \alpha, B_R} \leq C \ve^{3-m} |\log \ve| \| \phi_2 \|_{**, m-2, \alpha} \leq ( C\ve^{3-m} M |\log \ve | ) \, \ve |\log \ve|^{1\over 2} .
\ee
Observe also that $c_0[ H(\phi_1+\phi_2, \psi,\alpha) +B(\phi_2) ]$
\be\label{f33}
|c_0[ H(\phi_1+\phi_2, \psi,\alpha) +B(\phi_2) ]|\leq C \ve |\log \ve|^{1\over 2} (1+ \ve^{3-m} |\log \ve| \| \phi_2 \|_{**, m-2, \alpha} ).
\ee
From Propositions \ref{prop1}, \eqref{f2}-\eqref{f3} we conclude that
\be \label{f4}
\| \phi_1 \|_{*, m-2, B_R} \leq C \ve |\log \ve|^{1\over 2},
\ee
while from Lemma \ref{lemat} and \eqref{f33} we get
\be \label{f5}
\| \phi_2 \|_{**, m-2, \alpha} \leq C \ve |\log \ve|^{1\over 2}.
\ee
Moreover, using \eqref{c1}, we get that
\be \label{f6}
|\alpha_1 | \leq {C \over  |\log \ve |^{1\over 2}} .
\ee
Combining \eqref{f1}-\eqref{f4}-\eqref{f5}-\eqref{f6}, we conclude that ${\mathcal A} (\psi, \phi_1 , \phi_2 , \alpha_1 ) \in B$ if we choose $M$ large enough (but independently of $\ve$) in the definition of the set $B_M$ in \eqref{ball}.

We next show that ${\mathcal A}$ is a contraction map in $B_M$.
Let $\varphi_i = \eta_\delta (\phi_1^i + \phi_2^i) +\psi^i $, and $\alpha_1^i$ for $i=1,2$ such that
$(\psi^i, \phi_1^i, \psi_2^i ,\alpha_1^i ) \in B_M$. Let $G(\varphi^i, \alpha_1^i ) = G(\psi^i, \phi_1^i + \phi_2^i,\alpha_1^i)$
and observe that
\be \label{estGG}\begin{aligned}
\left| G(\varphi^1, \alpha_1^1 ) - G(\varphi^2, \alpha_1^2 )  \right|&\leq |V_{\alpha_1} (x) (\psi^1 - \psi^2)| \\
&+ |V_{\alpha_1} (x) - V_{\alpha_2} (x) | |\psi^1| + (1-\eta_\delta ) |N_{\Psi_{\alpha_1} } (\varphi^1 ) - N_{\Psi_{\alpha_2}} (\varphi^1) |\\
&+ (1-\eta_\delta ) | N_{\Psi_{\alpha_2}} (\varphi^1 ) - N_{\Psi_{\alpha_2}} (\varphi^2 ) |\\
&+|A[\phi_1^1-\phi_1^2]| + |A[\phi_2^1 - \phi_2^2]|
\end{aligned}
\ee
where $V_{\alpha} (x) =(1-\eta_\delta ) \ve^2 f'( \Psi_\alpha - \mu |x|^2 )$, $\mu = {\alpha \over 2} |\log \ve |$, and the other terms are defined in \eqref{GG}.
A direct computation gives that
\begin{align*}
 |V_{\alpha_1} (x) - V_{\alpha_2} (x) | |\psi^1| &+ (1-\eta_\delta ) |N_{\Psi_{\alpha_1} } (\varphi^1 ) - N_{\Psi_{\alpha_2}} (\varphi^1) | \\
 & \leq \ve^{1+b} e^{-|x|^2} |\alpha_1 - \alpha_2| \left( | \psi^1 | + |\varphi^1| \right),
 \end{align*}
\begin{align*}
|V_{\alpha_1} (x) (\psi^1 - \psi^2)|&+ (1-\eta_\delta ) | N_{\Psi_{\alpha_2}} (\varphi^1 ) - N_{\Psi_{\alpha_2}} (\varphi^2 ) | \\
&\leq \ve^{1+b } e^{-|x|^2}  \left( |\psi^1 - \psi^2| + \eta_\delta^2 (|\phi_1^1 - \phi_1^2|^2 + |\phi_2^1 - \phi_2^2|^2 ) \right)
\end{align*}
and
\begin{align*}
|A[\phi_1^1-\phi_1^2]|& \leq  C \left(|L_x (\eta_\delta ) |\phi_1^1 - \phi_1^2| + |K_{ij} \pp_i \eta_\delta \pp_j (\phi_1^2 - \phi_1^2) | \right) \\
&\leq { C \ve \over 1+ |x|^\nu } \| \phi_1^1 - \phi_1^2 \|_{*, m-2, B_R} .
\end{align*}
In order to estimate $A[\phi_2^1 - \phi_2^2]$, we observe that
\begin{align*}
\Delta_y & [\phi_2^1 - \phi_2^2]  +    f'(\Gamma_0) [\phi_2^1 - \phi_2^2] + c_0^{12}   e^{\Gamma_0} Z_0  = 0 \inn \R^2   \\
&{\mbox {where}} \quad c_0^{12}=   c_0 [H(\phi_1+\phi_2^1, \psi,\alpha) +B(\phi_2^1) ] - c_0 [H(\phi_1+\phi_2^2, \psi,\alpha) +B(\phi_2^2 )] .
\end{align*}
By definition,
$$
c_0^{12}= \int_{B_R} \left[ B[\phi_2^1 - \phi_2^1]  +   \mathcal N (\psi + \eta_\delta (\phi_1+ \phi_2^1)  ) - \mathcal N (\psi + \eta_\delta (\phi_1+ \phi_2^2)  ) \right] Z_0.
$$
Using \eqref{defB} and \eqref{house2}, we get
\begin{align*}
|c_0^{12} | &\leq C ( \int_{B_R} {\ve \over 1+|y|^2} \, dy ) \| \phi_2^1 - \phi_2^2\|_{**, \alpha} + C  \| \phi_2^1 - \phi_2^2\|_{**, \alpha}^2
\end{align*}
Thus we can conclude that
\begin{align*}
|A[\phi_2^1-\phi_2^2]|& \leq  C \left(|L_x (\eta_\delta ) |\phi_2^1 - \phi_2^2| + |K_{ij} \pp_i \eta_\delta \pp_j (\phi_2^2 - \phi_2^2) | \right) \\
&\leq { C \ve^\sigma \over 1+ |x|^\nu } \| \phi_1^1 - \phi_1^2 \|_{**, \alpha} .
\end{align*}
for some $\sigma >0$. Combining all these estimates in \eqref{estGG}, we obtain that
\begin{align*}
\left| G(\varphi^1, \alpha_1^1 ) - G(\varphi^2, \alpha_1^2 )  \right|&\leq
{C\ve^\sigma \over 1+|x|^\nu} \Bigl( \| \psi^1-\psi^2\|_\infty + \| \phi_1^1 - \phi_1^2 \|_{*,m-2, B_R} \\
& + \| \phi_1^1 - \phi_1^2 \|_{**,\alpha} + |\alpha_1^1 - \alpha_1^2 |\Bigl),
\end{align*}
which, in view of Proposition \ref{prop2}, gives
\begin{align*}
\| \TT^o [G(\varphi^1, \alpha_1^1 )] - \TT^o [G(\varphi^2, \alpha_1^2 )]  \|_\infty &\leq
C\ve^\sigma  \Bigl( \| \psi^1-\psi^2\|_\infty + \| \phi_1^1 - \phi_1^2 \|_{*,m-2, B_R} \\
& + \| \phi_1^1 - \phi_1^2 \|_{**,\alpha} + |\alpha_1^1 - \alpha_1^2 |\Bigl),
\end{align*}
for some $\sigma >0$.
Carrying out a similar analysis for each one of the operators in \eqref{fixp}, we get that ${\mathcal A}$ is a contraction mapping in $B_M$.

As a consequence,  Problem \eqref{fix}-\eqref{fixp} has a fixed point, which concludes the proof of Theorem \ref{teo1}. \qed

\section{Proof of Theorem \ref{teo2}}

The proof of Theorem \ref{teo2} is similar to that of Theorem \ref{teo1}, by employing the spherical symmetry of the equation (\ref{PB1}), which can be written in polar coordinates as
 \begin{equation}
     \label{PB100}
 {1\over h^2 + r^2} \left( {h^2 \over r^2} +1 \right) \pp_\theta^2 \Psi + {h^2 \over r} \pp_r \left( {r \over h^2 + r^2} \pp_r\Psi \right)=\ve^2 e^{(\Psi -{\alpha \over 2} |\log \ve | |x|^2 )} \quad {\mbox {in}} \quad \R^2.
\end{equation}

It is clear that equation (\ref{PB100}) is invariant under rotational symmetry $\theta \to \theta +\frac{2\pi}{k}$ and even symmetry $\theta \to -\theta$.

To construct multiple concentration solutions with polygonal symmetry, we work in the following space
\begin{equation}
\label{Xk} \mathcal{X}_k = \{ \Psi \in L^\infty_{loc} (\R^2) \ | \ \Psi ( z e^{\sqrt{-1} \frac{2\pi}{k}})=\Psi (z), \Psi (\bar{z})=\Psi (z), \ z =(z_1, z_2) \in \R^2  \}
\end{equation}
where $ k\geq 2$ is an integer.

As in the proof of Theorem \ref{teo1} we modify the approximate solution as follows: let $ \bar{H}_{1\ve} (z)=\Gamma_\ve (z)  \left( 1+ c_1 z_1 + c_2 |z|^2 \right) +{4 x_0^3 \over 1+ x_0^2} H_{1\ve} (z)$, where $\Gamma_\ve (z), c_1, c_2, x_0, H_{1\ve} (z)$ are defined in Section 3. Then the new approximate solution is
\begin{equation*}
%\label{psifinal100}
\begin{aligned}
\Psi_{ \alpha}  (x) &=  {\alpha \over 2} |\log \ve | x_0^2 - \log (1+ x_0^2)  \\
& \quad +    \sum_{j=1}^k \eta_\delta (x- x_0 e^{\sqrt{-1} \frac{2(j-1)\pi }{k}} ) \hat{H}_{1\ve} (z e^{\sqrt{-1} \frac{2(j-1)\pi}{k}} )   + H_{2\ve}  (x) \in \mathcal{X}_k .
 \end{aligned}
\end{equation*}
 Using the symmetry assumption (\ref{Xk}), the rest of the proof is exactly the same  as in that of Theorem \ref{teo1}. We omit the details.

\medskip\noindent{\bf Acknowledgments:}
J.~D\'avila has been supported  by  a Royal Society  Wolfson Fellowship, UK and  Fondecyt grant 1170224, Chile.
 M.~del Pino has been supported by a Royal Society Research Professorship, UK.
 M. Musso has been supported by EPSRC research Grant EP/T008458/1.
 The  research  of J.~Wei is partially supported by NSERC of Canada.

\end{document}